\documentclass[fleqn]{elsarticle}
\usepackage{amsmath,amssymb,amsthm,mathrsfs,graphicx,subfigure,float,caption,epstopdf,tabularx,bbm,color,bm,amsfonts,epic,booktabs,multirow}
\usepackage[colorlinks]{hyperref}

\usepackage{caption} 
\allowdisplaybreaks[4] 
\usepackage[toc,page]{appendix} 
\biboptions{numbers,sort&compress} 

\usepackage{geometry}
\makeatletter

\newcommand{\Rmnum}[1]{\expandafter\@slowromancap\romannumeral#1@}
\makeatother

\numberwithin{equation}{section}
\newtheorem{theorem}{Theorem}[section]

\newtheorem{remark}{Remark}[section]

\graphicspath{{Figs/}}

\begin{document}
\begin{frontmatter}
\title{Two novel classes of arbitrary high-order structure-preserving algorithms for canonical Hamiltonian systems\vspace{-1mm}}
\author{Yonghui Bo, Wenjun Cai, Yushun Wang$^*$}
\address{Jiangsu Key Laboratory for NSLSCS, Jiangsu Collaborative Innovation Center of Biomedial Functional Materials,\\ School of Mathematical Sciences, Nanjing Normal University, Nanjing 210023, China\vspace{-9mm}}

\begin{abstract}

In this paper, we systematically construct two classes of structure-preserving schemes with arbitrary order of accuracy for canonical Hamiltonian systems. The one class is the symplectic scheme, which contains two new families of parameterized symplectic schemes that are derived by basing on the generating function method and the symmetric composition method, respectively. Each member in these schemes is symplectic for any fixed parameter. A more general form of generating functions is introduced, which generalizes the three classical generating functions that are widely used to construct symplectic algorithms. The other class is a novel family of energy and quadratic invariants preserving schemes, which is devised by adjusting the parameter in parameterized symplectic schemes to guarantee energy conservation at each time step. The existence of the solutions of these schemes is verified. Numerical experiments demonstrate the theoretical analysis and conservation of the proposed schemes.
\end{abstract}
\end{frontmatter}

\begin{figure}[b]
\small \baselineskip=10pt
\rule[2mm]{1.8cm}{0.2mm} \par
$^{*}$Corresponding author.\\
E-mail address: wangyushun@njnu.edu.cn (Y. Wang).\\
\end{figure}
\vspace{-3mm}
\section{Introduction}
All real physical processes with negligible dissipation could be recast into the appropriate Hamiltonian form \cite{feng-85-difference}. The canonical Hamiltonian system with $d$ degrees of freedom has the following form
\begin{equation}\label{eq-1-1}
\dot{p} = -H_q(p,q),\quad \dot{q} = H_p(p,q),\quad p,q\in \mathcal{R}^d,
\end{equation}
where $H_p$ and $H_q$ denote the column vectors of partial derivatives of the Hamiltonian. And the Hamiltonian mostly represents the physical meaning of the total energy. The system \eqref{eq-1-1} serving as the basic mathematical formalism often appears in the relevant areas of analytical dynamics, geometrical optics, particle accelerators, plasma physics, control theory and hydrodynamics, etc. The most characteristic property of \eqref{eq-1-1} is the symplecticity and energy conservation along any exact flows \cite{feng-10-symplectic,hairer-06-geometric,sanzserna-94-Hamiltonian,luigi-16-line}. These facts motivate to search for numerical methods that can preserve one or more conserved properties. 

Numerical methods conserving the symplecticity are called symplectic integrators and generally have excellent behavior and stability in the long-term simulations \cite{hairer-06-geometric,sanzserna-94-Hamiltonian,mclachlan-11-stability-PRK}. Symplectic integrators were first introduced in the pioneering work \cite{vogelaere-56-contact} and developed in the early works \cite{ruth-83-canonical,menyuk-84-Hamiltonian,feng-85-difference,sanzserna-88-RK,qin-92-composition,marsden-98-multisymplectic}. Feng \cite{feng-85-difference} has obtained symplectic schemes by verifying that the one-step mapping of a numerical method is symplectic. There are several approaches to construct symplectic methods, such as symplectic Runge-Kutta (RK) methods \cite{sanzserna-88-RK,cooper-20-RK}, symplectic partitioned RK methods \cite{sun-93-SPRK,suris-90-RK-variational}, variational integrators \cite{marsden-98-multisymplectic} and the generating function method \cite{feng-89-generating}. The generating function method as a general framework has significance in both the theory and the construction of symplectic algorithms.

Using the generating function method and symmetric composition methods \cite{qin-92-composition}, we first develop two types of arbitrary high-order parameterized symplectic schemes. The free parameter in these schemes will be adjusted later to derive energy and quadratic invariants preserving (EQUIP) schemes \cite{luigi-12-EQUIP}. The complete theory of the generating function method has been established in \cite{feng-89-generating} as well as a series of applications \cite{hong-08-generating-MSRK,shang-94-volume,wang-14-stochastic}. The generating function is a solution of the Hamilton-Jacobi equation that can generate arbitrary symplectic mappings. It is possible to find such a function to generate a given symplectic transformation. Actually, for symplectic RK and partitioned RK methods, Hairer et al. \cite{hairer-06-geometric} have defined the expression of their generating functions. However, the construction of generating functions is dependent on the different choices of coordinates. In this paper, we introduce a generating function with new coordinates and discuss the associated Hamilton-Jacobi equation. The new coordinates unify and generalize the traditional three generating functions. Following the research route \cite{feng-89-generating,hairer-06-geometric}, we obtain the power series of the new generating function. Then, the first type of parameterized symplectic schemes is reported by truncating this series. In particular, a simple symplectic scheme with a free parameter $\lambda$ is obtained as
\begin{align}\label{eq-1-2}
\left\{\aligned
&p_{n+1} = p_n-hH_q\big(\lambda p_{n+1}+(1-\lambda)p_n,\lambda q_n+(1-\lambda)q_{n+1}\big),\\
&q_{n+1} = q_n+hH_p\big(\lambda p_{n+1}+(1-\lambda)p_n,\lambda q_n+(1-\lambda)q_{n+1}\big),\\
\endaligned\right.
\end{align}
where $h$ is the time step. We call the above scheme Scheme \Rmnum{1}. For the case where $\lambda=0$, $1$, $\frac{1}{2}$, respectively, the symplectic Euler methods and the implicit midpoint rule are recognized. Noting that the correct pairing of variables of the Hamiltonian plays a key role in the symplecticity of Scheme \Rmnum{1}. Substituting $1-\lambda$ for $\lambda$ yields its adjoint method, and the resulting method is still symplectic. By fixing the parameter to specific values, we can obtain more symplectic schemes with simple forms that can be widely applied in practical computations. The second type of parameterized symplectic schemes is derived by the composition of Scheme \Rmnum{1} and its adjoint method.

At present, the research on energy-preserving schemes of \eqref{eq-1-1} is also a prominent project. For more details, please refer to  the discrete gradient methods \cite{mclachlan-99-discrete-gradient}, the averaged vector field (AVF) method \cite{quispel-08-AVF,cai-18-PAVF,li-16-AVF,wang-13-parametric-PRK}, time finite element methods \cite{betsch-00-time-FEM,tang-12-time-FEM} and the Hamiltonian boundary value methods \cite{luigi-10-HBVM,luigi-16-line}, etc. Although symplectic schemes are superior in long-time behavior, they fail to conserve non-quadratic invariants. It is natural to think about whether such a discretization scheme can be found to share both the symplecticity and energy conservation. It is verified that the constant time stepping scheme cannot inherit both features for general Hamiltonian systems \cite{ge-88-Lie-Poisson,chartier-06-algebraic}. This difficulty has been solved in a weaker sense to obtain symplectic-energy-momentum integrators \cite{kane-99-variational}, where the symplecticity is viewed in the space-time sense. Also at a weaker level, a constructive idea by introducing a free parameter in Gauss symplectic methods which can be properly tuned to enforce the energy conservation at each step is introduced in \cite{luigi-12-EQUIP}. Later on, this approach is refined in \cite{luigi-18-analysis-EQUIP,luigi-16-line,wang-13-parametric-PRK,li-20-EQUIP-constraint} and named the EQUIP method. The application on the Hamiltonian wave equations develops EQUIP multi-symplectic methods \cite{chen-20-EQUIP-MS}. In more detail, Brugnano et al. \cite{luigi-12-EQUIP} first construct a family of one-step methods $(p_{n+1},q_{n+1})=\Phi_h^\lambda(p_n,q_n)$ depending on a free parameter $\lambda$ such that each method in this family is symplectic for any fixed parameter, and that a special value of the parameter can be chosen depending on $p_n,q_n$ and $h$ with the energy conservation at the same time. Compared with the previous energy-preserving methods, the EQUIP method also maintains the quadratic invariants of \eqref{eq-1-1}. Based on the obtained parameterized symplectic schemes, this paper present a novel family of arbitrary high-order EQUIP schemes. The one of these schemes has the following simple form as
\begin{align}\label{eq-1-5}
\left \{\aligned
&p_{n+1} = p_n-hH_q\big(\lambda p_{n+1}+(1-\lambda)p_n,\lambda q_n+(1-\lambda)q_{n+1}\big),\\
&q_{n+1} = q_n+hH_p\big(\lambda p_{n+1}+(1-\lambda)p_n,\lambda q_n+(1-\lambda)q_{n+1}\big),\\
&H(p_{n+1},q_{n+1}) = H(p_n,q_n).
\endaligned\right.
\end{align} 
The solvability of all EQUIP schemes such as \eqref{eq-1-5} is confirmed. The result indicates that the free parameter varies around $1/2$ to make the energy preserved, which reveals the superior behavior of the midpoint method in most cases. Notice that in such a way the Hamiltonian in \eqref{eq-1-5} can be replaced with any non-quadratic invariant of \eqref{eq-1-1} to achieve other invariant conservation.

This paper is organized as follows. In Sections \ref{Sec-2} and \ref{Sec-3}, two classes of arbitrary high-order symplectic schemes with a free parameter are obtained, respectively. In particular, we present three symplectic schemes, one of which is symmetric. In addition, we give an alternative derivation of the symplecticity of Scheme \Rmnum{1} from symplectic partitioned RK methods. In Section \ref{Sec-4}, we provide a class of arbitrary high-order EQUIP schemes and the theoretical proof of the existence of the solutions. Numerical tests in Section \ref{Sec-5} support the theoretical analysis. Finally, some concluding remarks are given in Section \ref{Sec-6}.

\section{Parameterized symplectic schemes based on generating functions}\label{Sec-2}
In this section, we introduce a generating function with new coordinates. Following the classic framework of the generating function method, we develop a class of parameterized symplectic schemes.

\subsection{Generating function with new coordinates}\label{subS-2-1}
Consider a time interval $[t_0,t_1]$, and denote by $p$, $q \in\mathcal{R}^d$ the initial values at $t_0$. Likewise, the exact solution at $t_1$ is represented by $P$, $Q \in\mathcal{R}^d$. This implies that the mapping $(p,q) \mapsto (P,Q)$ is symplectic.

\begin{theorem}\label{Th-2-1}
Let $\lambda$ be a real number and the mapping $\varphi:(p,q) \mapsto (P,Q)$ be smooth, close to the identity. It is symplectic if and only if a function $S\big(\lambda P+(1-\lambda)p,\lambda q+(1-\lambda)Q\big)$ exists locally such that
\begin{equation}\label{eq-2-1}
(Q-q)^Td\big(\lambda P+(1-\lambda)p\big)-(P-p)^Td\big(\lambda q+(1-\lambda)Q\big) = dS.
\end{equation}
\end{theorem}

\begin{proof}
We write the Jacobian of $\varphi$ as a block matrix $\frac{\partial(P,Q)}{\partial(p,q)} = \left(\begin{array}{cc} P_p & P_q\\ Q_p & Q_q \end{array}\right)$. The symplecticity is equivalent to the following equations \cite{hairer-06-geometric,feng-85-difference}
\begin{equation}\label{eq-2-2}
P_p^TQ_p = Q_p^TP_p,\quad Q_q^TP_q = P_q^TQ_q,\quad Q_p^TP_q+I = P_p^TQ_q.
\end{equation}
Substituting $dP = P_pdp+P_qdq$, $dQ = Q_pdp+Q_qdq$ into the left-hand side of \eqref{eq-2-1}, we obtain
\begin{equation}\label{eq-2-3}
\left(\begin{array}{c}
\lambda P_p^T(Q-q)-(1-\lambda)Q_p^T(P-p)+(1-\lambda)(Q-q)\\
\lambda P_q^T(Q-q)-(1-\lambda)Q_q^T(P-p)-\lambda(P-p)
\end{array}\right)^T
\left(\begin{array}{c}
dp\\dq
\end{array}\right).
\end{equation}
In order to use the integrability lemma \cite{hairer-06-geometric}, one would expect the symmetry of the Jacobian of the coefficient of \eqref{eq-2-3}. We get the Jacobian matrix as
\begin{equation}\label{eq-2-4}
\begin{aligned}
&\lambda\left(\begin{array}{cc}
P_p^TQ_p & P_p^TQ_q\\
P_q^TQ_p+I & P_q^TQ_q
\end{array}\right)
-(1-\lambda)\left(\begin{array}{cc}
Q_p^TP_p & Q_p^TP_q+I\\
Q_q^TP_p & Q_q^TP_q
\end{array}\right)
+\lambda\sum_{i=1}^d (Q_i-q_i)\frac{\partial^2P_i}{\partial(p,q)^2}\\
&+\left(\begin{array}{cc}
(1-\lambda)(Q_p+Q_p^T) & -\lambda P_p^T+(1-\lambda)Q_q\\
-\lambda P_p+(1-\lambda)Q_q^T & -\lambda(P_q+P_q^T)
\end{array}\right)
-(1-\lambda)\sum_{i=1}^d (P_i-p_i)\frac{\partial^2Q_i}{\partial(p,q)^2}.
\end{aligned}
\end{equation}
Since the Hessian matrixes of $P_i$ and $Q_i$ are symmetric, the symmetry of \eqref{eq-2-4} is equivalent to \eqref{eq-2-2}. Hence there exists locally a function $\hat{S}(p,q)$ such that
\begin{equation*}
(Q-q)^Td\big(\lambda P+(1-\lambda)p\big)-(P-p)^Td\big((\lambda q+(1-\lambda)Q\big) = d\hat{S}.
\end{equation*}
Hereafter, we use the notations $u = \lambda P+(1-\lambda)p$, $v = \lambda q+(1-\lambda)Q$. An implicit function
\begin{align*}
\left\{\aligned
&F_1(u,v,p,q) \triangleq \lambda P(p,q)+(1-\lambda)p-u = 0,\\
&F_2(u,v,p,q) \triangleq (1-\lambda)Q(p,q)+\lambda q-v = 0
\endaligned\right.
\end{align*}
is defined with the Jacobian as
\begin{equation}\label{eq-2-5}
\frac{\partial(F_1,F_2)}{\partial(p,q)} =
\left(\begin{array}{cc}
I_d+\lambda\big(\frac{\partial P}{\partial p}-I_d\big) & \lambda\frac{\partial P}{\partial q}\\
(1-\lambda)\frac{\partial Q}{\partial p} & \frac{\partial Q}{\partial q}-\lambda\big(\frac{\partial Q}{\partial q}-I_d\big)
\end{array}\right),
\end{equation}
where $I_d$ is a $d$-dimensional identity matrix. Since $\varphi$ is close to the identity, \eqref{eq-2-5} is invertible. The implicit function theorem thus proves that $p$ and $q$ are both functions of the new coordinates $u$ and $v$. Setting $S(u,v) \triangleq \hat{S}(p(u,v),q(u,v))$, this completes the proof.
\end{proof}

\begin{remark}\label{Re-2-1}
Theorem \ref{Th-2-1} gives a more general form of generating functions. When $\lambda = 0$, $1$, $\frac{1}{2}$, respectively, we reproduce the generating functions under the three classical coordinates \cite{hairer-06-geometric,feng-89-generating}, which are widely used to construct symplectic algorithms. For each fixed parameter, we can get a specific generating function.
\end{remark}

Comparing the coefficients of $dS = \partial_uSdu+\partial_vSdv$ with \eqref{eq-2-1}, we derive
\begin{align}\label{eq-2-6}
\left \{\aligned
&P = p-\partial_vS(\lambda P+(1-\lambda)p,\lambda q+(1-\lambda)Q),\\
&Q = q+\partial_uS(\lambda P+(1-\lambda)p,\lambda q+(1-\lambda)Q).
\endaligned\right.
\end{align}
Theorem \ref{Th-2-1} tells us that whatever the scalar function $S$ is, the system \eqref{eq-2-6} defines a symplectic transformation $(p, q) \mapsto (P,Q)$ for any fixed parameter $\lambda$.

\subsection{The Hamilton-Jacobi equation with new variables}\label{subS-2-2}
The Hamilton-Jacobi equation for a new time-dependent generating function $S(u,v,t)$ is introduced. Assuming the point $\big(P(t), Q(t)\big)$ to move along the exact flow of \eqref{eq-1-1}, it is seen that a smooth generating function $S(u,v,t)$ generates the exact flow by \eqref{eq-2-6} . 

\begin{theorem}\label{Th-2-2}
Let $\lambda$ be a real number. If $S(u,v,t)$ is a smooth solution of the Hamilton-Jacobi equation 
\begin{equation}\label{eq-2-7}
\frac{\partial S}{\partial t}(u,v,t) = H\left(u-(1-\lambda)\frac{\partial S}{\partial v}(u,v,t), v+\lambda\frac{\partial S}{\partial u}(u,v,t)\right)
\end{equation}
with initial condition $S(u,v,0) = 0$, then the mapping $(p,q) \mapsto (P,Q)$ defined by \eqref{eq-2-6}, is the exact flow of the Hamiltonian system \eqref{eq-1-1}.
\end{theorem}

\begin{proof}
Differentiating the first equation of \eqref{eq-2-6} with respect to time $t$ yields
\begin{equation}\label{eq-2-8}
\frac{\partial^2 S}{\partial t\partial v}(u,v,t) = -\left(I_d+\lambda\frac{\partial^2 S}{\partial u\partial v}(u,v,t)\right)\dot{P}-(1-\lambda)\frac{\partial^2 S}{\partial v^2}(u,v,t)\dot{Q}.
\end{equation}
Differentiating both sides of \eqref{eq-2-7} about the variable $v$ reads
\begin{equation}\label{eq-2-9}
\frac{\partial^2 S}{\partial t\partial v}(u,v,t) = -(1-\lambda)\frac{\partial^2 S}{\partial v^2}(u,v,t)H_P(P,Q)+\left(I_d+\lambda\frac{\partial^2 S}{\partial u\partial v}(u,v,t)\right)H_Q(P,Q).
\end{equation}
Subtracting \eqref{eq-2-9} from \eqref{eq-2-8}, one obtain
\begin{equation}\label{eq-2-10}
\left(I_d+\lambda\frac{\partial^2 S}{\partial u\partial v}(u,v,t)\right)\big(\dot{P}+H_Q(P,Q)\big)+(1-\lambda)\frac{\partial^2 S}{\partial v^2}(u,v,t)\big(\dot{Q}-H_P(P,Q)\big) = 0.
\end{equation}
Similarly, differentiating the second equation of \eqref{eq-2-6} and using \eqref{eq-2-7} lead to
\begin{equation}\label{eq-2-11}
-\lambda\frac{\partial^2 S}{\partial u^2}(u,v,t)\big(\dot{P}+H_Q(P,Q)\big)+\left(I_d-(1-\lambda)\frac{\partial^2 S}{\partial u\partial v}(u,v,t)\right)\big(\dot{Q}-H_P(P,Q)\big) = 0.
\end{equation}
Combining \eqref{eq-2-10} and \eqref{eq-2-11}, we get the matrix equation
\begin{equation}\label{eq-2-12}
\left (\begin{array}{cc}
I_d+\lambda\frac{\partial^2 S}{\partial u\partial v}(u,v,t) & (1-\lambda)\frac{\partial^2 S}{\partial v^2}(u,v,t)\\
-\lambda\frac{\partial^2 S}{\partial u^2}(u,v,t) & I_d-(1-\lambda)\frac{\partial^2 S}{\partial u\partial v}(u,v,t)
\end{array}\right)
\left(\begin{array}{c}
\dot{P}+H_Q(P,Q)\\ \dot{Q}-H_P(P,Q)
\end{array}\right) = 0.
\end{equation}
Owing to the mapping $(p,q) \mapsto (P,Q)$ close to the identity, the coefficient matrix of \eqref{eq-2-12} is non-singular. This proves the validity of two equations of \eqref{eq-1-1}, i.e.,
\begin{equation*}
\dot{P} = -H_Q(P,Q),\quad \dot{Q} = H_P(P,Q).
\end{equation*}
The condition $S(u,v,0) = 0$ means that $P(0) = p$, $Q(0) = q$.
\end{proof}

\begin{remark}\label{Re-2-2}
Based on Theorem \ref{Th-2-2}, when $\lambda = 0$, $1$, $\frac{1}{2}$, respectively, we can recover the Hamilton-Jacobi equations corresponding to the three classical variables \cite{hairer-06-geometric,feng-10-symplectic,feng-89-generating}. This generalization of the Hamilton-Jacobi equation will directly lead to the following parameterized symplectic schemes.
\end{remark}

\subsection{Symplectic schemes based on generating functions}\label{subS-2-3}
A series of symplectic schemes can be constructed by basing on the three classical generating functions \cite{hairer-06-geometric,feng-89-generating}. We still consider employing an approximate solution of \eqref{eq-2-7} to construct symplectic schemes. For this purpose, a power series about $t$ is as follows:
\begin{equation}\label{eq-2-13}
S(u,v,t) = \sum_{i=1}^{\infty} K_i(u,v)t^i,
\end{equation}
which is an explicit expression of the new generating function with undetermined $K_i$, and satisfies $S(u,v,0) = 0$. Inserting this series into \eqref{eq-2-7} and comparing like powers of $t$ obtain
\begin{align*}
&K_1(u,v) = H(u,v),\\
&K_2(u,v) = (\lambda-\frac{1}{2})\left(\big(\frac{\partial H}{\partial u}\big)^T\frac{\partial H}{\partial v}\right)(u,v),\\
&K_3(u,v) = \frac{1}{2}(\lambda^2-\lambda+\frac{1}{3})\left(\big(\frac{\partial H}{\partial u}\big)^T \frac{\partial^2H}{\partial v^2}\frac{\partial H}{\partial u}+\big(\frac{\partial H}{\partial v}\big)^T \frac{\partial^2H}{\partial u^2}\frac{\partial H}{\partial v}\right)(u,v)
\nonumber\\&~~~~~~~~~~~~~~+(\lambda^2-\lambda+\frac{1}{6})\left(\big(\frac{\partial H}{\partial v}\big)^T\frac{\partial^2H}{\partial u\partial v}\frac{\partial H}{\partial u}\right)(u,v),
\end{align*}
and further $K_i(u,v)$, $i \geqslant 4$ can be obtained by computing the Taylor expansion. A natural way to approximate $S$ is to truncate the series \eqref{eq-2-13} by
\begin{equation}\label{eq-2-14}
\bar{S}(u,v) = \sum_{i=1}^{r} K_i(u,v)h^i.
\end{equation}
Inserting \eqref{eq-2-14} into \eqref{eq-2-6}, when $\lambda\neq\frac{1}{2}$, we get a symplectic scheme of order $r$. For all the following schemes, we use the notations
\begin{equation}\label{eq-2-15}
\bar{u} = \lambda p_{n+1}+(1-\lambda)p_n,\quad \bar{v} = \lambda q_n+(1-\lambda)q_{n+1}.
\end{equation}
Next, two symplectic schemes for \eqref{eq-1-1} are presented, which contain a free parameter $\lambda$. These two schemes can be obtained by taking $r =1$, $2$ in \eqref{eq-2-14}, respectively.
\begin{align}\left \{\aligned
&p_{n+1} = p_n-h\partial_{\bar{v}}H(\bar{u},\bar{v}),\\
&q_{n+1} = q_n+h\partial_{\bar{u}}H(\bar{u},\bar{v}).\\
\endaligned\right.
\end{align}
\vspace{-4mm}
\begin{align}
\left \{\aligned
&p_{n+1} = p_n-h\partial_{\bar{v}}H(\bar{u},\bar{v})-(\lambda-\frac{1}{2})h^2\left(\frac{\partial^2H}{\partial {\bar{v}}^2}\frac{\partial H}{\partial {\bar{u}}}+\frac{\partial^2H}{\partial {\bar{v}}\partial {\bar{u}}}\frac{\partial H}{\partial {\bar{v}}}\right)(\bar{u},\bar{v}),\\
&q_{n+1} = q_n+h\partial_{\bar{u}}H(\bar{u},\bar{v})+(\lambda-\frac{1}{2})h^2\left(\frac{\partial^2H}{\partial {\bar{u}}^2}\frac{\partial H}{\partial {\bar{v}}}+\frac{\partial^2H}{\partial {\bar{u}}\partial {\bar{v}}}\frac{\partial H}{\partial {\bar{u}}}\right)(\bar{u},\bar{v}).\\
\endaligned\right.
\end{align}

The former is Scheme \Rmnum{1}, and we name the latter Scheme \Rmnum{2}. Keeping $K_3$, a symplectic scheme of order $3$ can be gained. This procedure can be repeated to obtain parameterized symplectic schemes of any high order. It is noted that for $r \geqslant 3$, these schemes require the computation of higher derivatives. 

\begin{remark}\label{Re-2-3}
Among all known symplectic methods, only the symplectic Euler methods and the implicit midpoint rule have the simplest form, and they are very friendly in practical applications \cite{miniati-15-galaxies,hairer-06-geometric,sanzserna-94-Hamiltonian}. Here, these two methods are unified, so that for any fixed parameter, we can obtain a simple symplectic scheme.
\end{remark}

\subsection{Derivation of Scheme \Rmnum{1} based on symplectic partitioned RK methods}\label{subS-2-4}
We attempt to give an alternative derivation of Scheme \Rmnum{1} by partitioned RK methods. Appling partitioned RK methods \cite{hairer-06-geometric,sanzserna-94-Hamiltonian} to \eqref{eq-1-1} generates
\begin{align}\label{eq-2-18}
\left\{\aligned
&k_i=-H_q\Big(p_n+h\sum_{j=1}^{s}a_{ij}k_j,q_n+h\sum_{j=1}^{s}\hat{a}_{ij}l_j\Big),\ l_i=H_p\Big(p_n+h\sum_{j=1}^{s}a_{ij}k_j,q_n+h\sum_{j=1}^{s}\hat{a}_{ij}l_j\Big),\\
&p_{n+1}=p_n+h\sum_{i=1}^{s}b_ik_i,\ q_{n+1}=q_{n}+h\sum_{i=1}^{s}\hat{b}_il_i,
\endaligned\right.
\end{align}
where $b_i$, $a_{ij}$ and $\hat{b}_i$, $\hat{a}_{ij}$ $(i,j=1,\dots,s)$ are the coefficients of two RK methods with $s$ stages. It is known that if the identities 
\begin{equation}\label{eq-2-19}
b_i\hat{a}_{ij}+\hat{b}_ja_{ji}=b_i\hat{b}_j,\ b_i=\hat{b}_i\quad \mbox{for}\ i,j=1,\dots,s
\end{equation}
hold, then \eqref{eq-2-18} is symplectic \cite{hairer-06-geometric,sun-93-SPRK}. We present a one-stage partitioned RK method by setting $a_{11}=\lambda$, $\hat{a}_{11}=1-\lambda$ and $b_1=\hat{b}_{1}=1$. Its Butcher tableau is of the form
$$\begin{tabular}{c|c}
$\lambda$ & $\lambda$\\ \hline
& $1$\\
\end{tabular}\quad
\begin{tabular}{c|c}
$1-\lambda$ & $1-\lambda$\\ \hline
& $1$\\
\end{tabular}$$
with any parameter $\lambda$. Moreover, this method applied to \eqref{eq-1-1} can be written as
\begin{align}\label{eq-2-20}
\left\{\aligned
&k_1=-H_q\big(p_n+\lambda hk_1,q_n+(1-\lambda)hl_1\big),\ l_1=H_p\big(p_n+\lambda hk_1,q_n+(1-\lambda)hl_1\big),\\
&p_{n+1}=p_n+hk_1,\ q_{n+1}=q_{n}+hl_1.
\endaligned\right.
\end{align}
Notice that \eqref{eq-2-20} is exactly Scheme \Rmnum{1} through a simple check. Scheme \Rmnum{1} satisfies the condition \eqref{eq-2-19} such that it is proved as a symplectic scheme for any $\lambda$. 

\begin{remark}\label{Re-2-4}
Actually, the symplecticity of Scheme \Rmnum{1} can be verified by the following equation \cite{feng-85-difference} as
\begin{equation}\label{eq-2-21}
\bigg(\frac{\partial(p_{n+1},q_{n+1})}{\partial(p_n,q_n)}\bigg)^TJ\bigg(\frac{\partial(p_{n+1},q_{n+1})}{\partial(p_n,q_n)}\bigg)= J,\quad J = \left(\begin{array}{cc} 0 & I_d\\ -I_d & 0 \end{array}\right).
\end{equation} 
In order to save space, this detail can be omitted.
\end{remark}

\begin{remark}\label{Re-2-5}
For a separable system with the kinetic energy $T(p)=\frac{1}{2}p^TM^{-1}p$ and the potential $U(q)$, where $M$ is a constant non-degenerate matrix, the dynamics obeys $\dot{p} = -U_q(q)$, $\dot{q} = M^{-1}p$, which is equivalent to the second-order system $\ddot{q} = -M^{-1}U_q(q)$. Then one can obtain a one-stage symplectic Nystr\"{o}m method by applying Scheme \Rmnum{1} as
\begin{align*}
\left\{\aligned
&k_1=-M^{-1}U_q\big(q_n+(1-\lambda)h\dot{q}_n+\lambda(1-\lambda)h^2k_1\big),\\
&q_{n+1}=q_{n}+h\dot{q}_n+\lambda h^2k_1,\ \dot{q}_{n+1}=\dot{q}_n+hk_1.
\endaligned\right.
\end{align*}
This is the Nystr\"{o}m method of Scheme \Rmnum{1}. It is explicit if and only if $\lambda=0$ or $1$, which is the symplectic Euler methods for separable systems.
\end{remark}

\section{Parameterized symplectic schemes based on symmetric composition methods}\label{Sec-3}
With the purpose of avoiding higher derivatives of the Hamiltonian in the generating function method, we consider the symmetric composition methods \cite{qin-92-composition,hairer-06-geometric}. Scheme \Rmnum{1} denotes a numerical flow
\begin{equation}\label{eq-3-1}
\Phi_h^{\lambda}: (p_n,q_n) \mapsto (p_{n+1},q_{n+1}).
\end{equation}
By a straightforward check, the adjoint method of scheme \Rmnum{1} is
\begin{equation*}
(\Phi_h^{\lambda})^* = \Phi_h^{1-\lambda}.
\end{equation*}
Noting that $\Phi_h^{1-\lambda}$ is also symplectic for any parameter $\lambda$. Hence the symmetry requirement for order $2$ leads to a symmetric numerical flow, denoted by
\begin{equation}\label{eq-3-2}
\Psi_h^\lambda\triangleq \Phi_{\frac{h}{2}}^{\lambda} \circ (\Phi_{\frac{h}{2}}^{\lambda})^* = \Phi_{\frac{h}{2}}^{\lambda} \circ \Phi_{\frac{h}{2}}^{1-\lambda}: (p_n,q_n) \mapsto (p_{n+1},q_{n+1}),
\end{equation}
which is named Scheme \Rmnum{3}. The symplecticity of Scheme \Rmnum{3} follows that the composition of symplectic schemes is still symplectic. 

Using the idea of composition methods, we can derive a symmetric symplectic scheme of order $4$ by the symmetric composition of Scheme \Rmnum{3}. Just as the Triple Jump and Suzuki's Fractals \cite{hairer-06-geometric,suzuki-90-decomposition}, this procedure can be repeated to construct higher order symplectic schemes without higher derivatives. Here, we present Scheme \Rmnum{3} as an example of this class of parameterized symplectic schemes.

\section{The construction and solvability of EQUIP schemes}\label{Sec-4}
Symplectic schemes only conserve the quadratic invariants \cite{cooper-20-RK,hairer-06-geometric,luigi-16-line}. Here, we concentrate on the case that the Hamiltonian is not quadratic. Following the idea in \cite{luigi-12-EQUIP}, the parameter $\lambda$ may be properly tuned to achieve energy conservation at each time step.

\subsection{Energy-preserving schemes}\label{subS-4-1}
We construct a novel family of EQUIP methods by using the parameterized symplectic schemes obtained in Section \ref{Sec-3}. Based on Scheme \Rmnum{1}, we first get the following EQUIP scheme named Scheme \Rmnum{4}. The notations $\bar{u}$ and $\bar{v}$ are introduced in \eqref{eq-2-15}.
\begin{align*}
 \left \{
 \aligned
 &p_{n+1} = p_n-h\partial_{\bar{v}}H(\bar{u},\bar{v}),\\
 &q_{n+1} = q_n+h\partial_{\bar{u}}H(\bar{u},\bar{v}),\\
 &H(p_{n+1},q_{n+1}) = H(p_n,q_n).
 \endaligned
 \right.
\end{align*}

\begin{remark}\label{Re-4-1}
If a real value $\lambda$ can be found such that $H(p_{n+1},q_{n+1}) = H(p_n,q_n)$, we get a scheme \eqref{eq-3-1} at the current step, which ensures energy conservation. It is known that for a fixed $\lambda$, \eqref{eq-3-1} is symplectic. If the existence of the parameter in Scheme \Rmnum{4} is verified, there exists a sequence \{$\lambda_i$\} as the diagram
\begin{equation}\label{eq-4-1}
(p_n,q_n) \stackrel{\Phi_h^{\lambda_1}}{\longmapsto} (p_{n+1},q_{n+1}) \stackrel{\Phi_h^{\lambda_2}}{\longmapsto} (p_{n+2},q_{n+2}) \stackrel{\Phi_h^{\lambda_3}}{\longmapsto} (p_{n+3},q_{n+3}) \longmapsto \cdots
\end{equation}
such that the solution $(p_{n+i},q_{n+i})$ defined by $\Phi_h^{\lambda_i}$ satisfies $H(p_{n+i},q_{n+i})=H(p_n,q_n)$, $i=1,2,\cdots$. The existence of \{$\lambda_i$\} will be proved later. Noting that in \eqref{eq-4-2}, symplectic schemes with different fixed parameters are used to obtain energy conservation. Since the value $\lambda_i$ is generally not equal to $1/2$, Scheme \Rmnum{4} preserves all quadratic invariants of the form $p^TCq$ with a constant matrix $C$.
\end{remark}

Similarly, the following EQUIP scheme called Scheme \Rmnum{5} is obtained by Scheme \Rmnum{3}, where the notation $\Psi_h^\lambda$ is used in \eqref{eq-3-2}.
\begin{align*}
 \left \{
 \aligned
 &(p_{n+1},q_{n+1}) = \Psi_h^\lambda(p_n,q_n),\\
 &H(p_{n+1},q_{n+1}) = H(p_n,q_n).
 \endaligned
 \right.
\end{align*}
For Schemes \Rmnum{4} and \Rmnum{5}, we need to seek a fixed value of the parameter $\lambda$ at each step such that $H(p_{n+1},q_{n+1}) = H(p_n,q_n)$. Scheme \Rmnum{4} is an energy-preserving scheme of order $1$, while Scheme \Rmnum{5} is a second order symmetric scheme. 

Based on the obtained schemes in Section \ref{Sec-3}, higher order EQUIP schemes can be derived by using $\lambda$ to enforce energy conservation. These schemes accurately conserve the quadratic invariants of the form $p^TCq$. In terms of computational complexity, we do not consider the construction by employing the high-order parameterized symplectic schemes of the generating function method.

\subsection{The solvability of the EQUIP schemes}\label{subS-4-2}
A novel family of EQUIP methods is constructed, which includes the above Schemes \Rmnum{4} and \Rmnum{5} as its members. Moreover, the solvability of these schemes is a necessary problem. For the given $(p_n,q_n)$ and $h$, we first consider the existence of the solution of the following equation
\begin{equation*}
G(p_{n+1},q_{n+1},\lambda,p_n,q_n,h)\triangleq
\left(\begin{array}{c}
p_{n+1}-p_n+hH_q\big(\lambda p_{n+1}+(1-\lambda)p_n,\lambda q_n+(1-\lambda)q_{n+1}\big)\\
q_{n+1}-q_n-hH_p\big(\lambda p_{n+1}+(1-\lambda)p_n,\lambda q_n+(1-\lambda)q_{n+1}\big)\\
H(p_{n+1},q_{n+1})-H(p_n,q_n)
\end{array}\right)=0
\end{equation*}
which is equivalent to Scheme \Rmnum{4}. The function $G$ vanishes at point $(p_n,q_n,\frac{1}{2},p_n,q_n,0)$. The Jacobian of $G$ is obtained as
\begin{equation}\label{eq-4-2}
\frac{\partial G}{\partial(p_{n+1},q_{n+1},\lambda)}\bigg|_{(p_n,q_n,\frac{1}{2},p_n,q_n,0)} =
\left ( \begin{array}{ccc}
I_d & 0 & 0\\
0 & I_d & 0\\
H_p(p_n,q_n) & H_q(p_n,q_n) & 0
\end{array}\right).
\end{equation}
Since \eqref{eq-4-2} is degenerate anyway, the traditional steps does not help to acquire the solvability of Scheme \Rmnum{4}. Thus the Newton-type iteration is invalid. For this reason, the scalar function
\begin{equation}\label{eq-4-3}
E(\lambda,h)=H\big(p_{n+1}(\lambda,h),q_{n+1}(\lambda,h)\big)-H(p_n,q_n)
\end{equation}
is introduced as an error of the Hamiltonian, where $p_{n+1}(\lambda,h)$ and $q_{n+1}(\lambda,h)$ are defined by the first two equations of Scheme \Rmnum{4}. At this moment, the solvability of Scheme IV can be obtained from the existence of a solution in the form $\lambda=\lambda(h)$ of $E(\lambda,h)=0$ with $E(\frac{1}{2},0)=0$ as the initial condition. 

The following derivatives of \eqref{eq-4-3} vanish at $(\frac{1}{2},0)$, i.e.,
\begin{equation}\label{eq-4-4}
E_{\lambda}(\frac{1}{2},0)=E_h(\frac{1}{2},0)=E_{\lambda\lambda}(\frac{1}{2},0)=E_{\lambda h}(\frac{1}{2},0)=E_{hh}(\frac{1}{2},0)=0.
\end{equation}
This is an unexpected result making the solvability difficult, but \eqref{eq-4-4} prompt us to continue the succedent derivation.
  
\begin{theorem}\label{Th-4-1}
Assuming that the Hamiltonian is sufficiently smooth, then there exists a function $\lambda$=$\lambda(h)$ defined in a neighborhood $(-h_0,h_0)$ with $h_0$ small enough, such that
\begin{enumerate}
\item $E\big(\lambda(h),h\big)=0$ for all $h \in (-h_0,h_0)$;
\item $\lambda(h)=\frac{1}{2}+\varepsilon(h)h$ with $\varepsilon(h)=\text{constant}+\mathcal{O}(h)$.
\end{enumerate}
\end{theorem} 

\begin{proof}
The series \eqref{eq-2-14} indicates that when $\lambda=\frac{1}{2}$, Scheme \Rmnum{1} is of order $2$, otherwise it is of order $1$. Thus one can obtain
\begin{align}\label{eq-4-5}
E(\lambda,h) = \left\{\aligned
&C_0h^3+\mathcal{O}(h^4),\ C_0\neq0,\quad\quad\quad\mbox{for}\ \lambda=\frac{1}{2},\\
&C(\lambda)h^2+\mathcal{O}(h^3),\ C(\lambda)\neq0,\quad\mbox{for}\ \lambda\neq\frac{1}{2},
\endaligned\right.
\end{align}
where the constant $C_0$ and the function $C(\lambda)$ also depend on $p_n$ and $q_n$. Then the expansion of \eqref{eq-4-4} can be written as
\begin{equation}\label{eq-4-6}
E(\lambda,h) = \sum_{i\geq3}\frac{1}{i!}\frac{\partial^iE}{\partial h^i}\left(\frac{1}{2},0\right)h^i+\sum_{i\geq2}\sum_{j\geq1}\frac{1}{i!j!}\frac{\partial^{i+j}E}{\partial h^i\partial \lambda^j}\left(\frac{1}{2},0\right)h^i\left(\lambda-\frac{1}{2}\right)^j.
\end{equation}
We search for the existence of a solution of the form $\lambda(h)=\frac{1}{2}+\varepsilon(h)h$, where $\varepsilon(h)$ is a real-valued function of $h$. Substituting $\lambda=\frac{1}{2}+\varepsilon h$ into \eqref{eq-4-6} gives
\begin{equation*}
E(\lambda,h) = \sum_{i\geq3}\frac{1}{i!}\frac{\partial^iE}{\partial h^i}\left(\frac{1}{2},0\right)h^i+\sum_{i\geq2}\sum_{j\geq1}\frac{1}{i!j!}\frac{\partial^{i+j}E}{\partial h^i\partial \lambda^j}\left(\frac{1}{2},0\right)h^{i+j}\varepsilon^j.
\end{equation*}
When $h\neq0$, $E(\lambda,h)=0$ is equivalent to $\bar{E}(\varepsilon,h)=0$, where
\begin{equation}\label{eq-4-7}
\bar{E}(\varepsilon,h)=\frac{1}{6}\frac{\partial^3E}{\partial h^3}\left(\frac{1}{2},0\right)+\frac{1}{2}\frac{\partial^3E}{\partial h^2\partial \lambda}\left(\frac{1}{2},0\right)\varepsilon+\mathcal{O}(\varepsilon h)+\mathcal{O}(h).
\end{equation}
Let the initial point
\begin{equation}\label{eq-4-8}
\varepsilon_0=-\frac{1}{3}\frac{\partial^3E}{\partial h^3}\left(\frac{1}{2},0\right)\Big/\frac{\partial^3E}{\partial h^2\partial \lambda}\left(\frac{1}{2},0\right),\quad h_0=0,
\end{equation}
then $\bar{E}(\varepsilon_0,h_0)=0$. Moreover, the equation \eqref{eq-4-5} shows that $\frac{\partial^3E}{\partial h^3}\left(\frac{1}{2},0\right)$ and $\frac{\partial^3E}{\partial h^2\partial \lambda}\left(\frac{1}{2},0\right)$ are not equal to zero. According to
\begin{equation*}
\frac{\partial \bar{E}}{\partial \varepsilon}(\varepsilon_0,h_0)=\frac{1}{2}\frac{\partial^3E}{\partial h^2\partial \lambda}\left(\frac{1}{2},0\right)\neq0,
\end{equation*}
employing the implicit function theorem to \eqref{eq-4-7} reports that the existence and uniqueness of $\varepsilon=\varepsilon(h)$ in $h \in (-h_0,h_0)$ with $h_0$ small enough are guaranteed, which reads
\begin{equation}\label{eq-4-9}
\varepsilon(h)=\varepsilon_0+\mathcal{O}(h)=\mbox{constant}+\mathcal{O}(h)
\end{equation}
ensuring that $\bar{E}(\varepsilon(h),h)=0$ holds. Therefore, the existence of $\lambda(h)=\frac{1}{2}+\varepsilon(h)h$ around $\frac{1}{2}$ is obtained to make Scheme \Rmnum{4} solvable. In addition, using the notations \eqref{eq-3-1} and \eqref{eq-3-2}, one can derive 
\begin{equation}\label{eq-4-10}
\Psi_h^\lambda(p_n,q_n)-\Phi_h^{\lambda}(p_n,q_n)=C_1(\lambda)h^l+\mathcal{O}(h^{l+1}), \quad l\geq2,
\end{equation}
where $l\geq3$ for $\lambda=\frac{1}{2}$. In this sense, the Hamiltonian error of Scheme \Rmnum{5} reads
\begin{equation}\label{eq-4-12}
E(\lambda,h)=H\big(\Phi_h^{\lambda}(p_n,q_n)+C_1(\lambda)h^l+\mathcal{O}(h^{l+1})\big)-H(p_n,q_n).
\end{equation}
The above derivation is valid for verifying the solvability of Scheme \Rmnum{5}. Similarly, we can obtain the same conclusions for higher order EQUIP schemes.
\end{proof}

In the sense of Theorem \ref{Th-4-1}, Schemes \Rmnum{4} and \Rmnum{5} can be regarded as a small perturbation of the midpoint method, which obtains energy conservation. The efficient solution of \eqref{eq-4-3} will be the part of future investigations. In the numerical tests, we use a standard solver for parameterized symplectic schemes coupled with the secant iteration for determining the parameter $\lambda$ that achieves energy conservation.

\section{Numerical experiments}\label{Sec-5} 
In this section, we present two numerical examples to confirm the theoretical analysis and conservation of the proposed schemes. The symplectic schemes with respect to four different parameter values are tested. In order to compare with the EQUIP methods, we calculate numerical solutions by the AVF method \cite{quispel-08-AVF}. The time step is set to $h=0.02$. The convergence rates of all the schemes are obtained by
\begin{equation*}
\mbox{Order} = \log_2 \frac{\mbox{Error}_h}{\mbox{Error}_{h/2}},\ \mbox{Error}_h = ||y_h(1)-y_{h/2}(1)||_\infty,
\end{equation*}
where $y_h(1)$ represents the numerical solutions of corresponding schemes at time $t=1$.

\subsection{The H\'{e}non-Heiles model}\label{subS-5-1}
Our first experiment is the H\'{e}non-Heiles model that originates from a problem in celestial mechanics. After a reduction of the dimension, this model is equivalent to the motion of a test particle with unit mass subject to an arbitrary potential $U(q_1,q_2)$ in the $(q_1,q_2)$-plane. The dynamics is described by a Hamiltonian of the form
\begin{equation}\label{eq-5-1}
H(p_1,p_2,q_1,q_2) = \frac{1}{2}(p_1^2+p_2^2)+U(q_1,q_2).
\end{equation}
The potential $U(q_1,q_2) = \frac{1}{2}(q_1^2+q_2^2+2 q_1^2 q_2-\frac{2}{3}q_2^3)$ is chosen for this experiment.

The particular potential $U=\frac{1}{6}$ serving as the critical energy of this model consists of three straight lines forming an equilateral triangle, whose vertices are saddle points of $U$ as plotted in Fig. \ref{Fig-1}. For values of energy $H<\frac{1}{6}$, the equipotential curves of the system are close thus making escape impossible. However, for larger energy levels ($H>\frac{1}{6}$), the equipotential curves open and three exit channels appear through which the test particles may escape to infinity \cite{zotos-15-HH}. The two classical orbits of the H\'{e}non-Heiles model are simulated with the initial values as
\begin{enumerate}
\item Box orbits (see the first row of Fig. \ref{Fig-1}):
\[H_0=0.02,\ p_2(0)=0,\ q_1(0)=0,\ q_2(0)=-0.082;\]
\item Chaotic orbits (see the second row of Fig. \ref{Fig-1}):
\[H_0=\frac{1}{6},\ p_2(0)=0,\ q_1(0)=0,\ q_2(0)=0.82.\]
\end{enumerate}
The values of $p_1(0)$ are found from the Hamiltonian \eqref{eq-5-1}.

The box orbits emerging with the internal shape of an equilateral triangle are presented in Fig. \ref{Fig-1}. Under the critical energy, Schemes \Rmnum{1}--\Rmnum {3} depict that the numerical trajectories never escapes to the outside of the equilateral triangle. The orbits of three symplectic schemes are similar, so we only show the ones of Scheme \Rmnum{1}. The symplecticity of all schemes ensures that the energy error remains bounded and small over long times in Fig. \ref{Fig-2}. The convergence rates of Schemes \Rmnum{1}--\Rmnum {3} are reported in Tab. \ref{Tab-1}. 

The three energy-preserving schemes give the satisfactory orbits in Fig. \ref{Fig-3}. Fig. \ref{Fig-4} presents the energy error and the value distribution of $\lambda$ in Schemes \Rmnum{4} and \Rmnum{5}. These results illustrate that the parameter $\lambda$ fluctuate around $1/2$ to achieve energy conservation and confirm the theoretical analysis in Theorem \ref{Th-4-1}. The energy errors of Schemes \Rmnum{4} and \Rmnum{5} are smaller than the AVF method and remain bounded over long times. Tab. \ref{Tab-2} exhibits the convergence rates for three energy-preserving schemes.

\begin{figure}[H]
\centering
\subfigure[$\lambda=-1$ (first column)]{
\begin{minipage}[t]{0.23\textwidth}
\centering
\includegraphics[width=40mm]{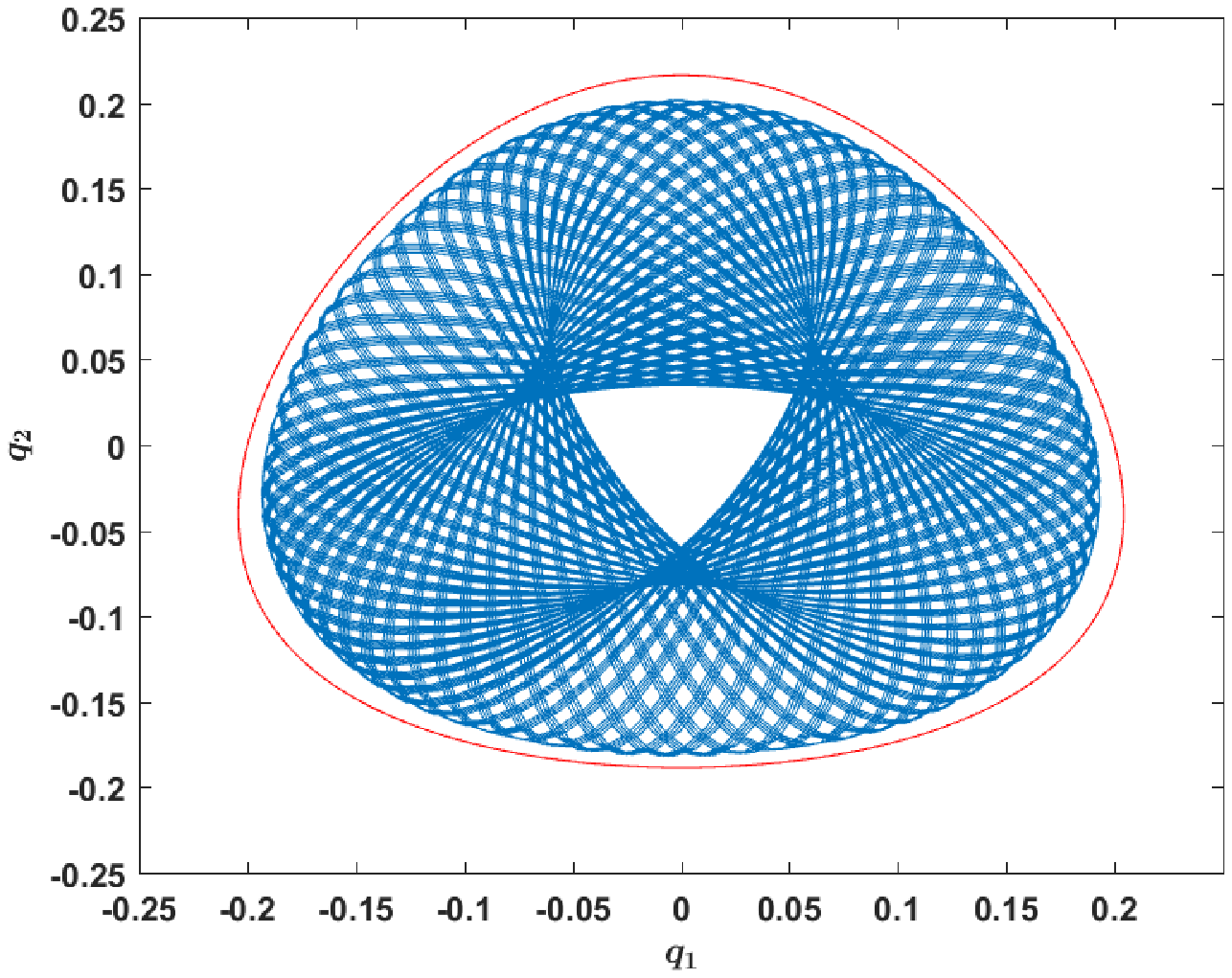}
\includegraphics[width=40mm]{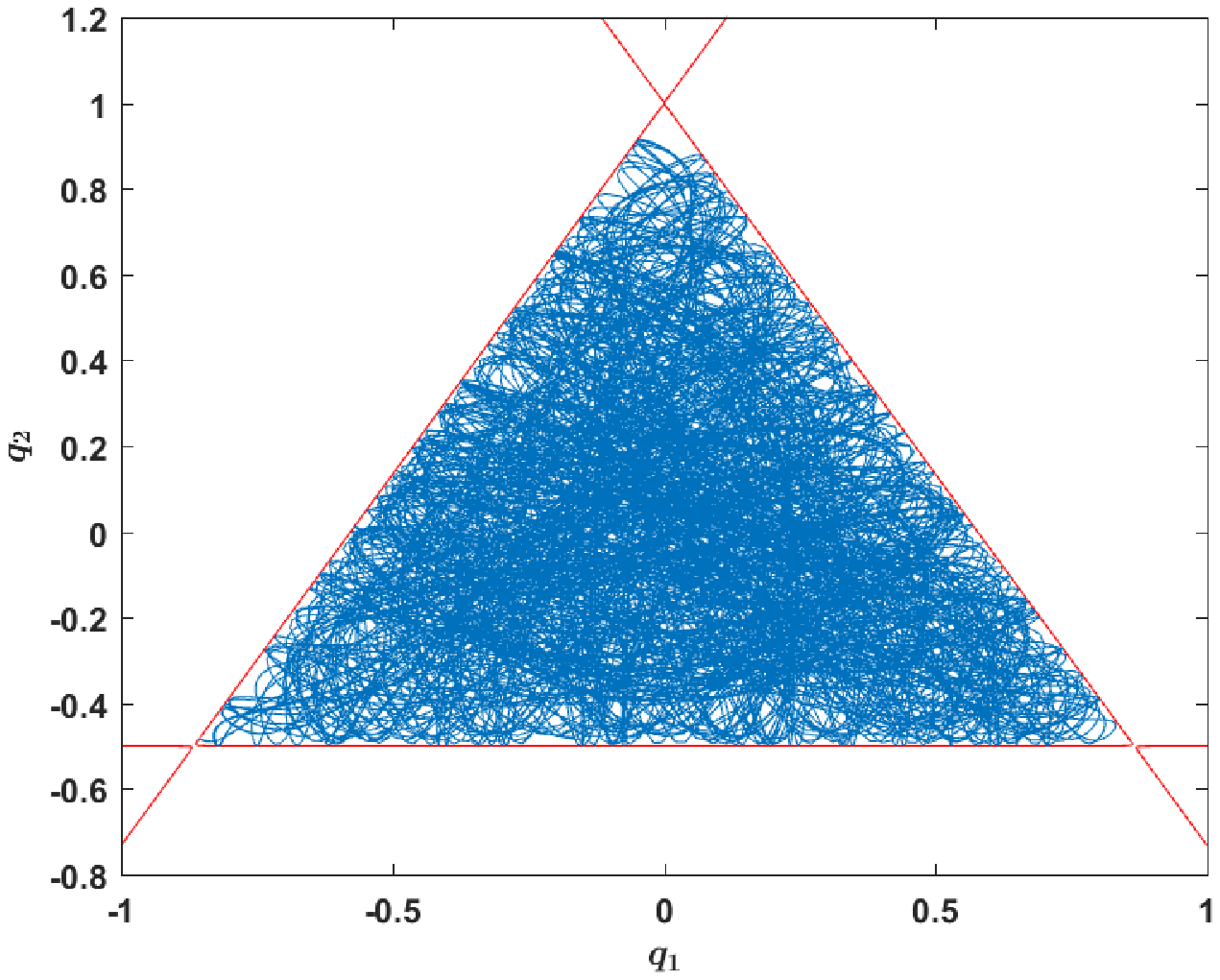}
\end{minipage}}\hspace{-2mm}
\subfigure[$\lambda=\frac{1}{3}$ (second column)]{
\begin{minipage}[t]{0.23\textwidth}
\centering
\includegraphics[width=40mm]{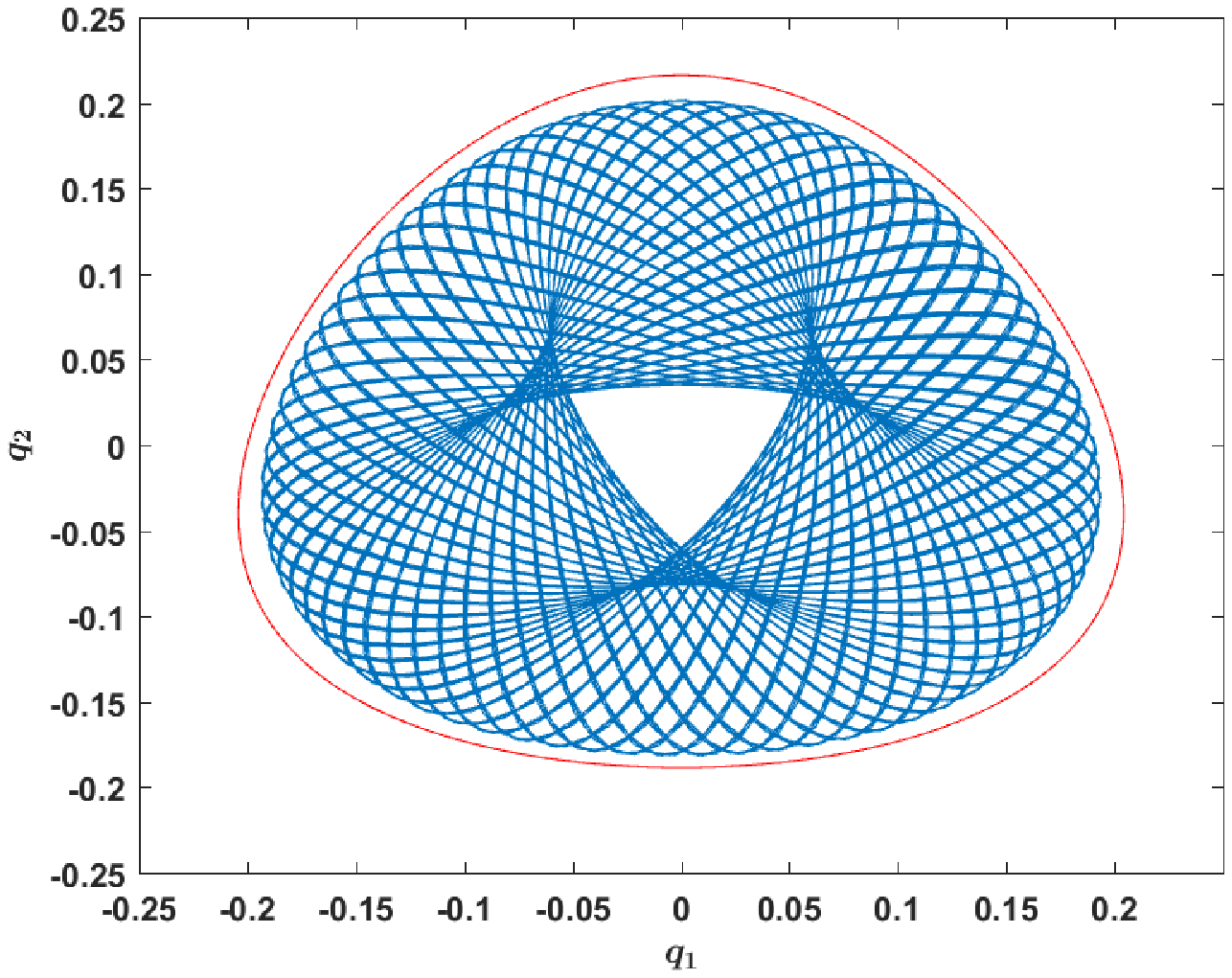}
\includegraphics[width=40mm]{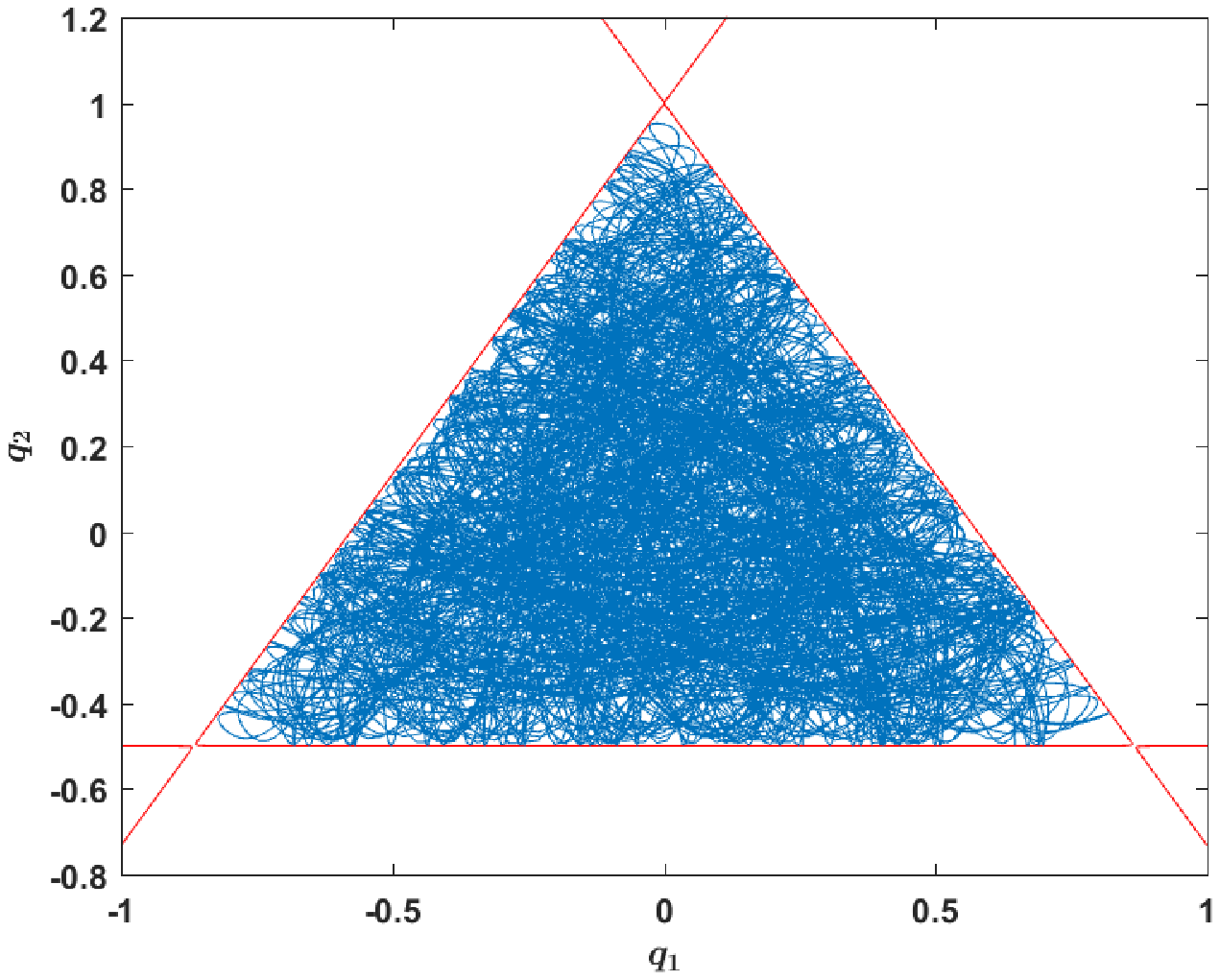}
\end{minipage}}\hspace{-2mm}
\subfigure[$\lambda=\frac{3}{2}$ (third column)]{
\begin{minipage}[t]{0.23\textwidth}
\centering
\includegraphics[width=40mm]{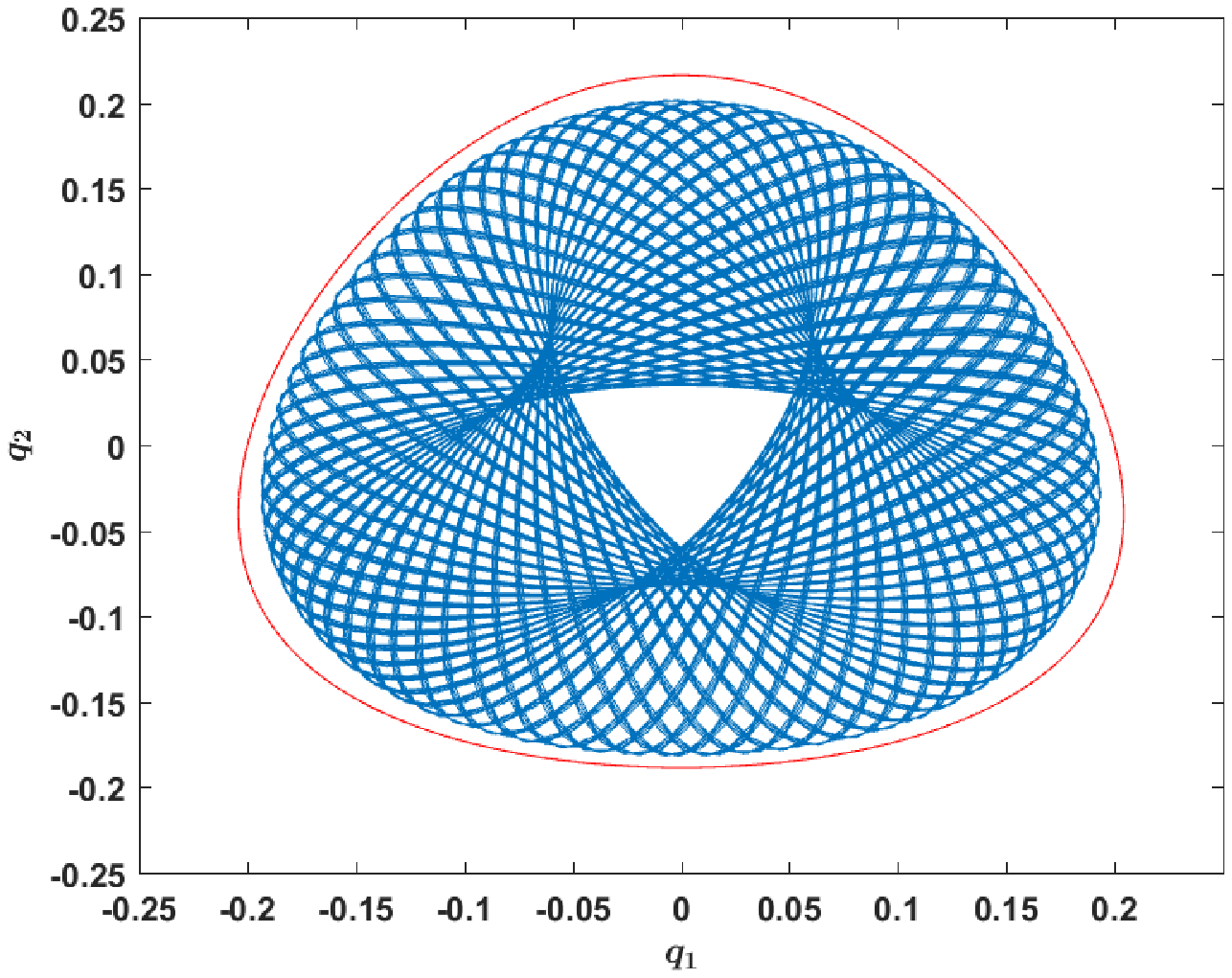}
\includegraphics[width=40mm]{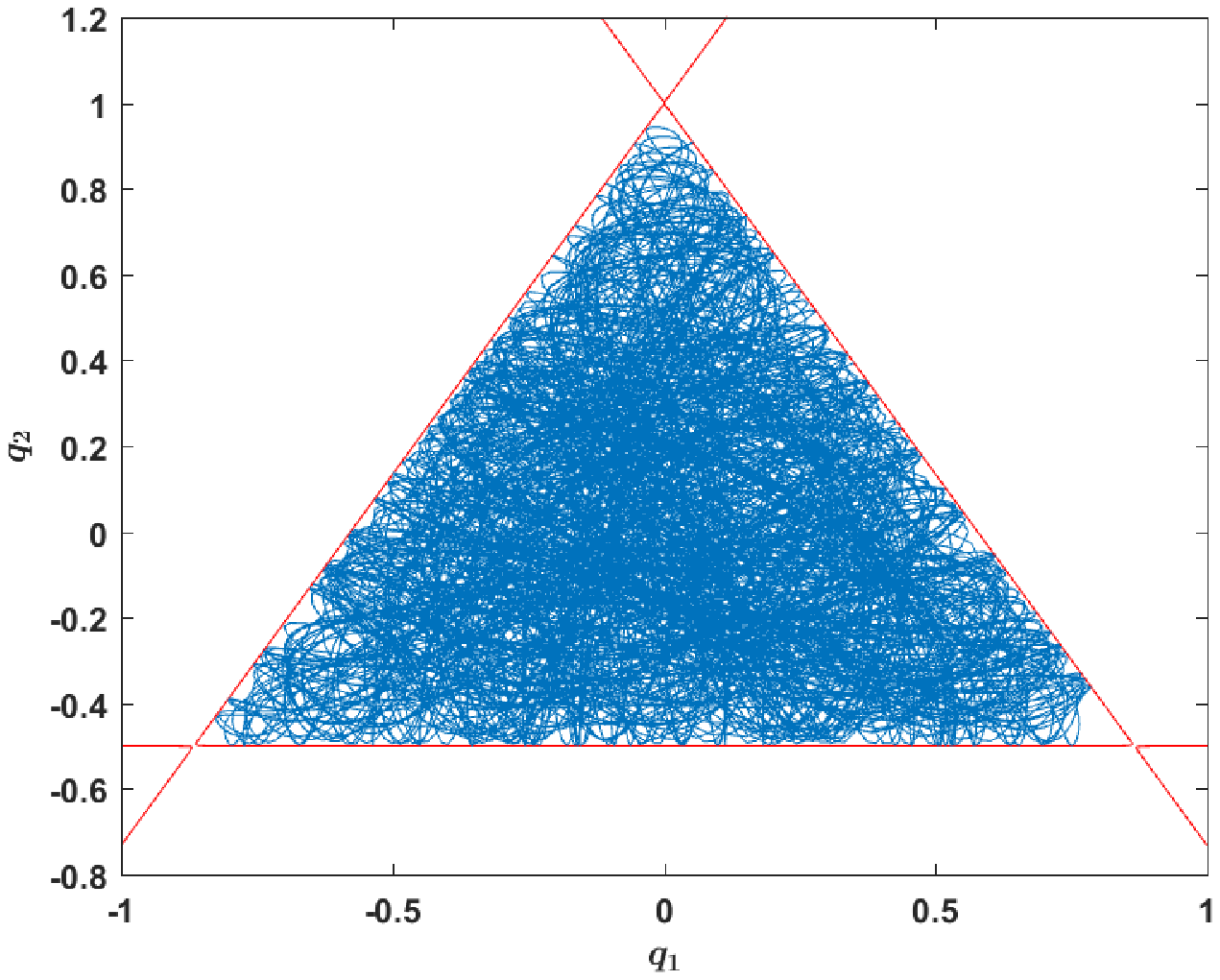}
\end{minipage}}\hspace{-2mm}
\subfigure[$\lambda=2$ (fourth column)]{
\begin{minipage}[t]{0.23\textwidth}
\centering
\includegraphics[width=40mm]{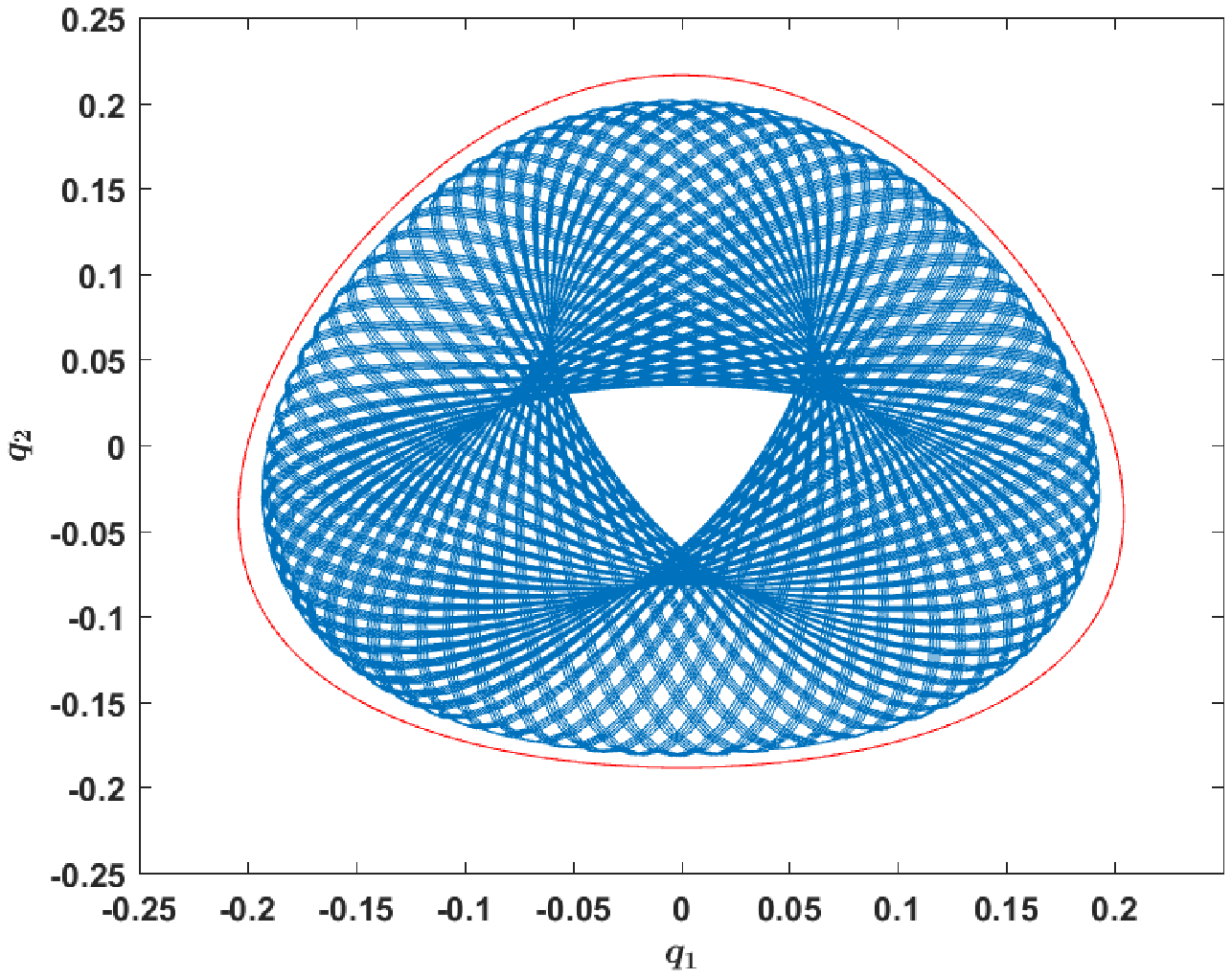}
\includegraphics[width=40mm]{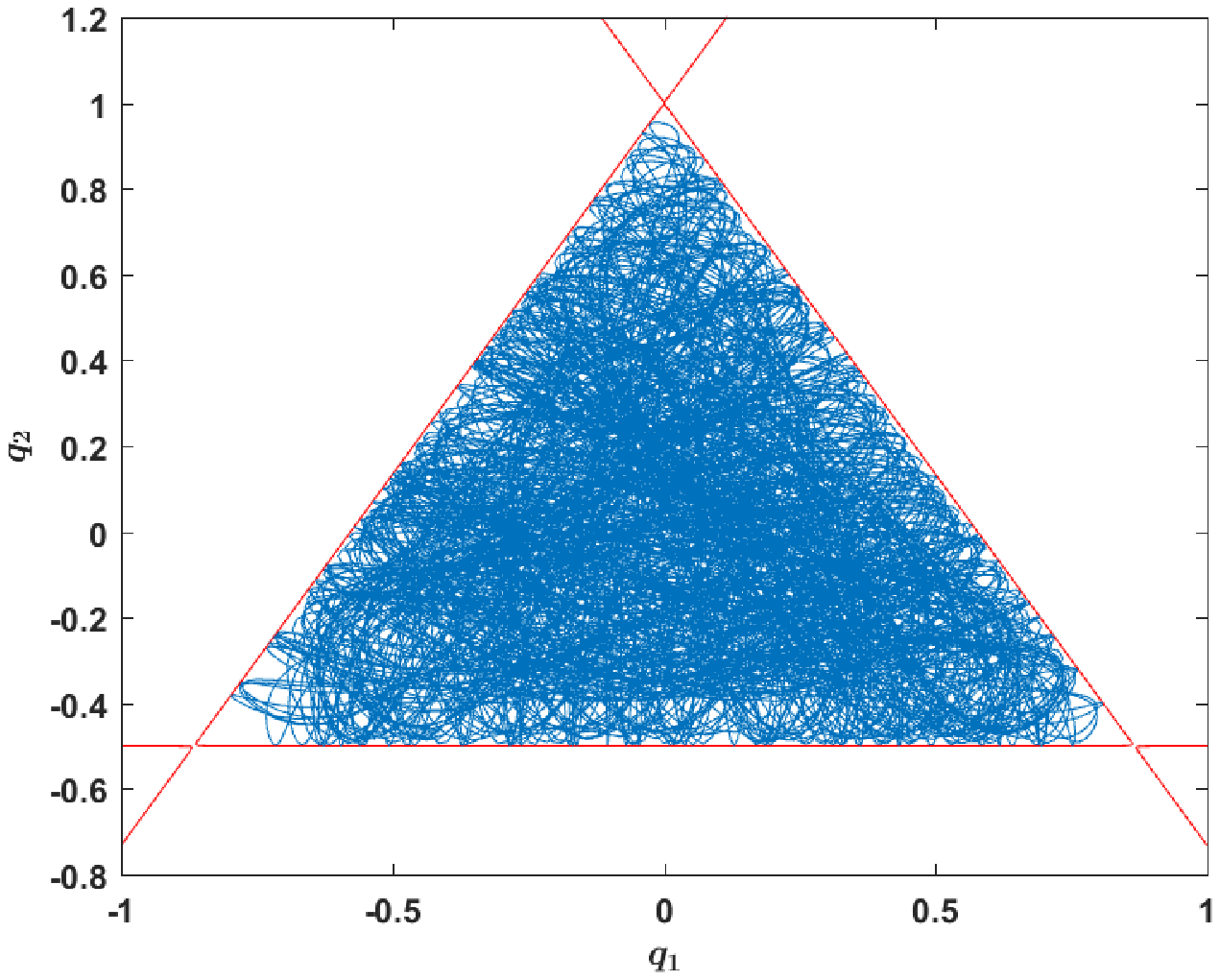}
\end{minipage}}
\caption{Box orbits (till $t = 1000$) and chaotic orbits (till $t = 2000$) of Scheme \Rmnum{1} for the H\'{e}non-Heiles model.}\label{Fig-1}
\end{figure} 

\begin{figure}[H]
\centering
\subfigure[Scheme \Rmnum{1} (first column)]{
\begin{minipage}[t]{0.32\textwidth}
\centering
\includegraphics[width=50mm]{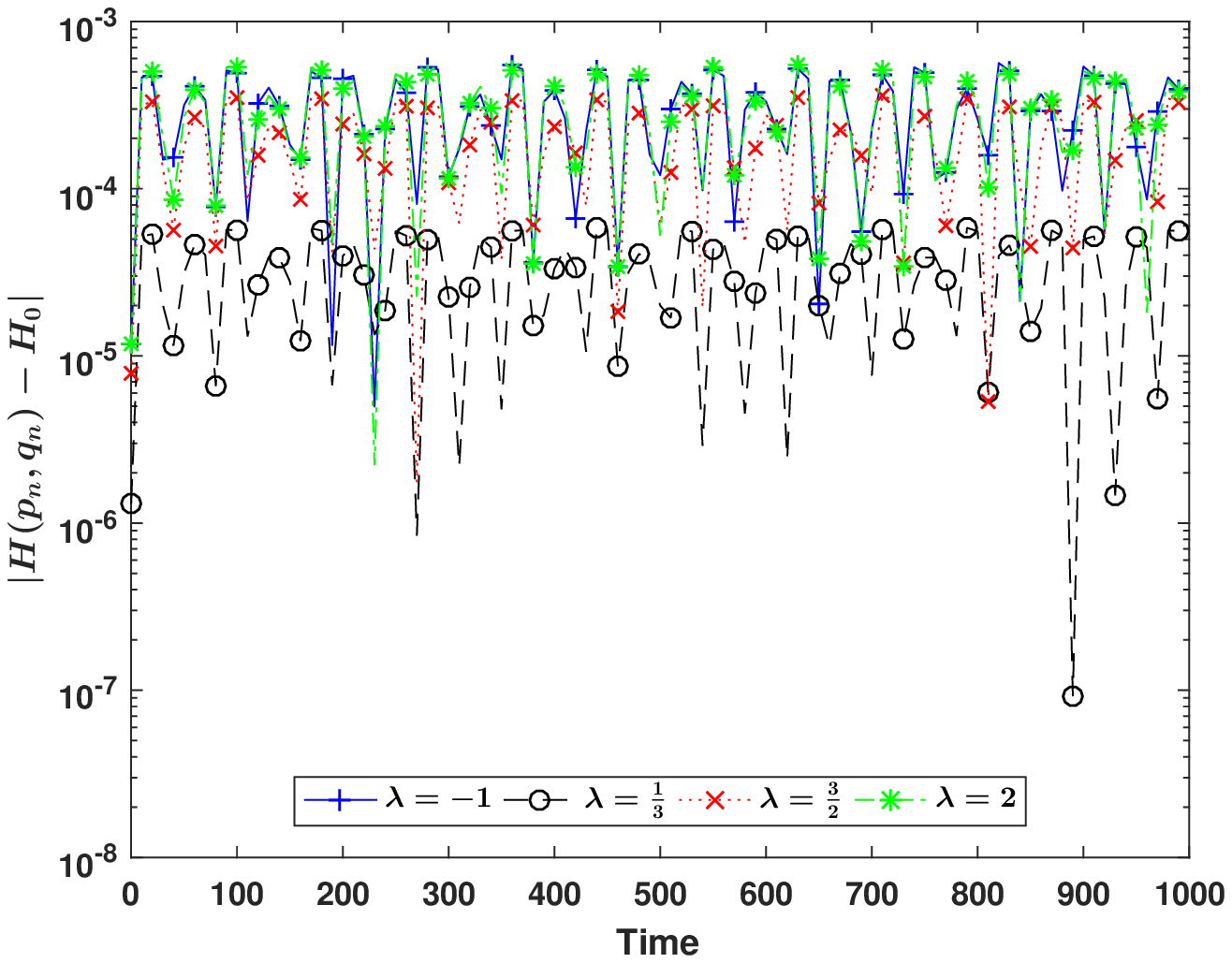}
\includegraphics[width=50mm]{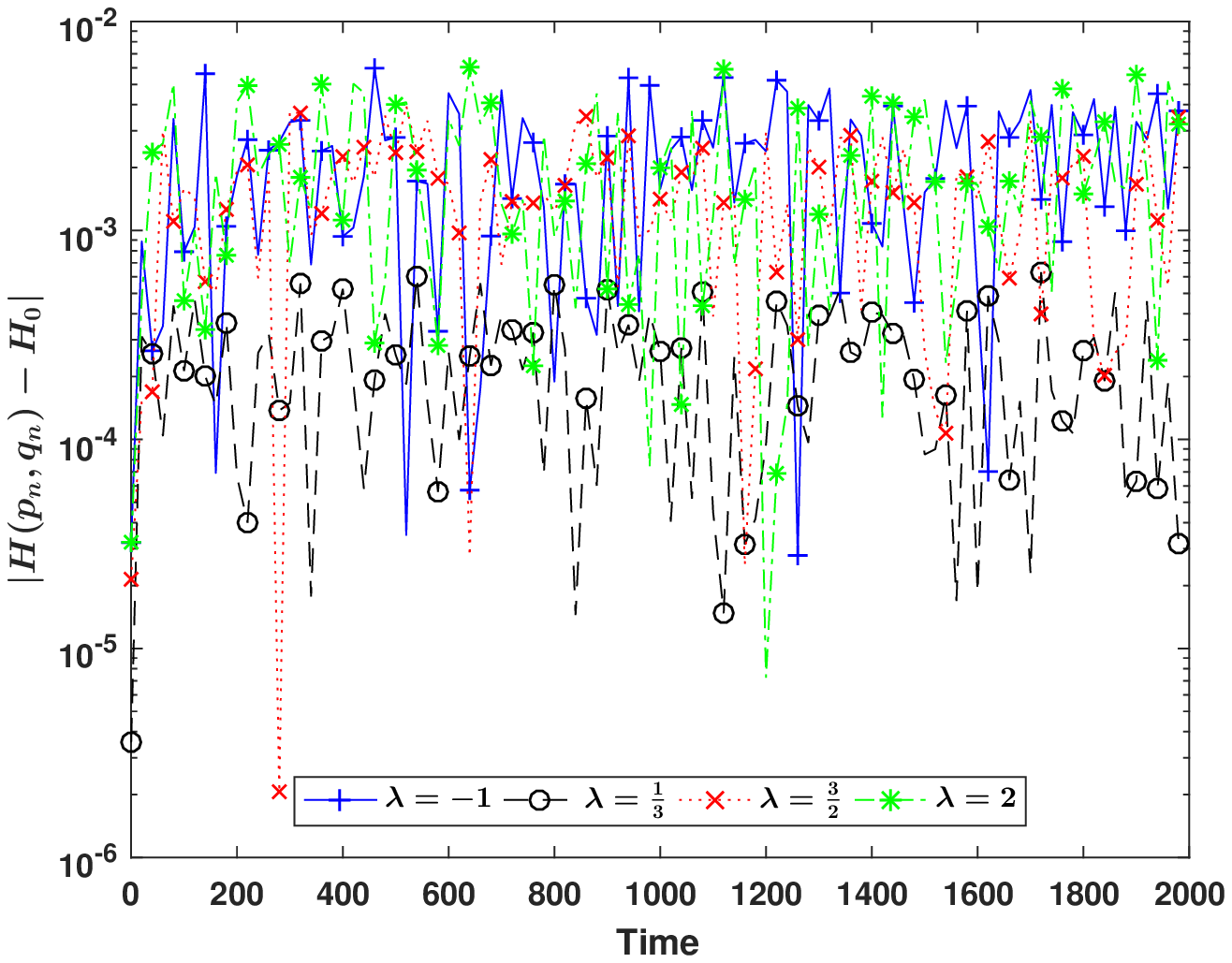}
\end{minipage}}\hspace{-6.5mm}
\subfigure[Scheme \Rmnum{2} (second column)]{
\begin{minipage}[t]{0.32\textwidth}
\centering
\includegraphics[width=50mm]{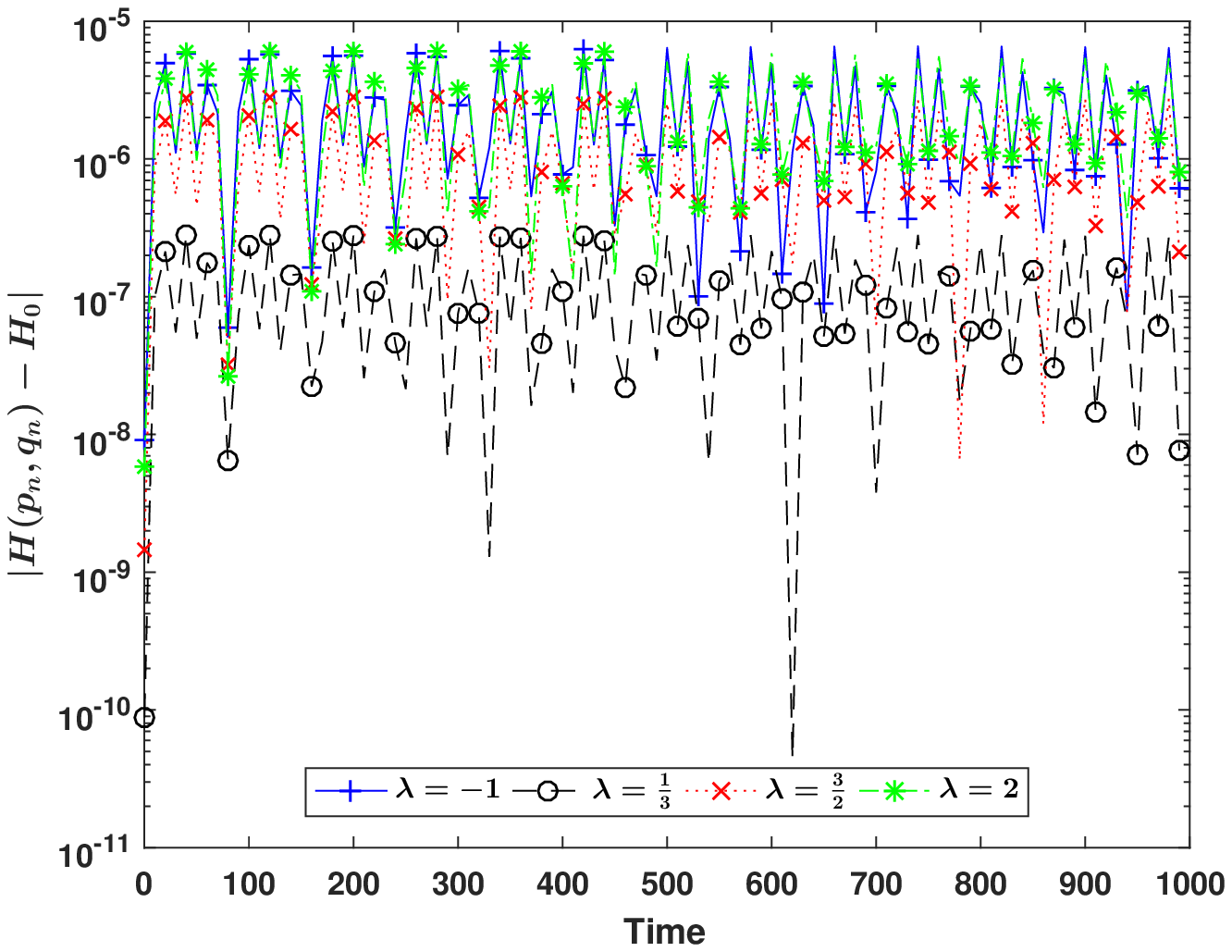}
\includegraphics[width=50mm]{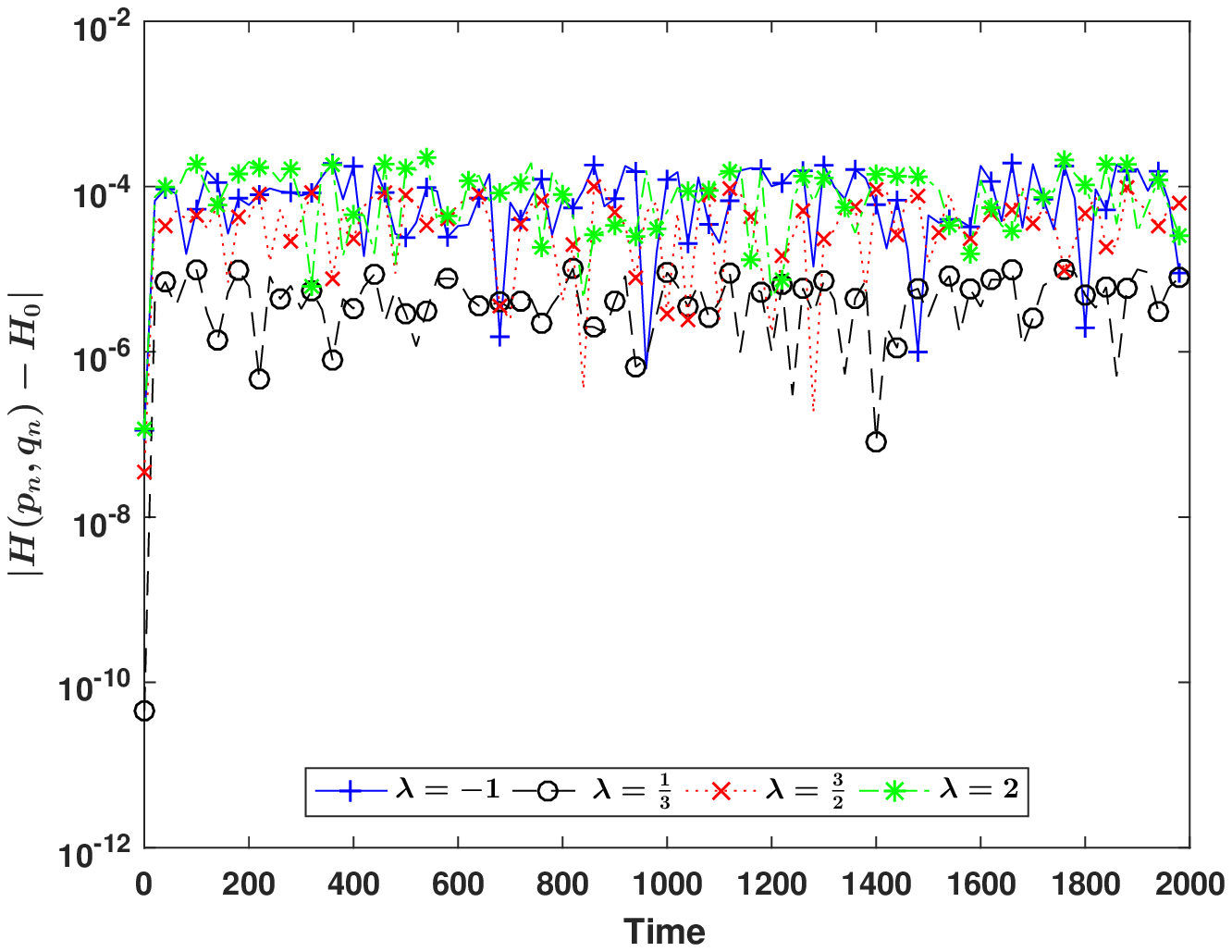}
\end{minipage}}\hspace{-6.5mm}
\subfigure[Scheme \Rmnum{3} (third column)]{
\begin{minipage}[t]{0.32\textwidth}
\centering
\includegraphics[width=50mm]{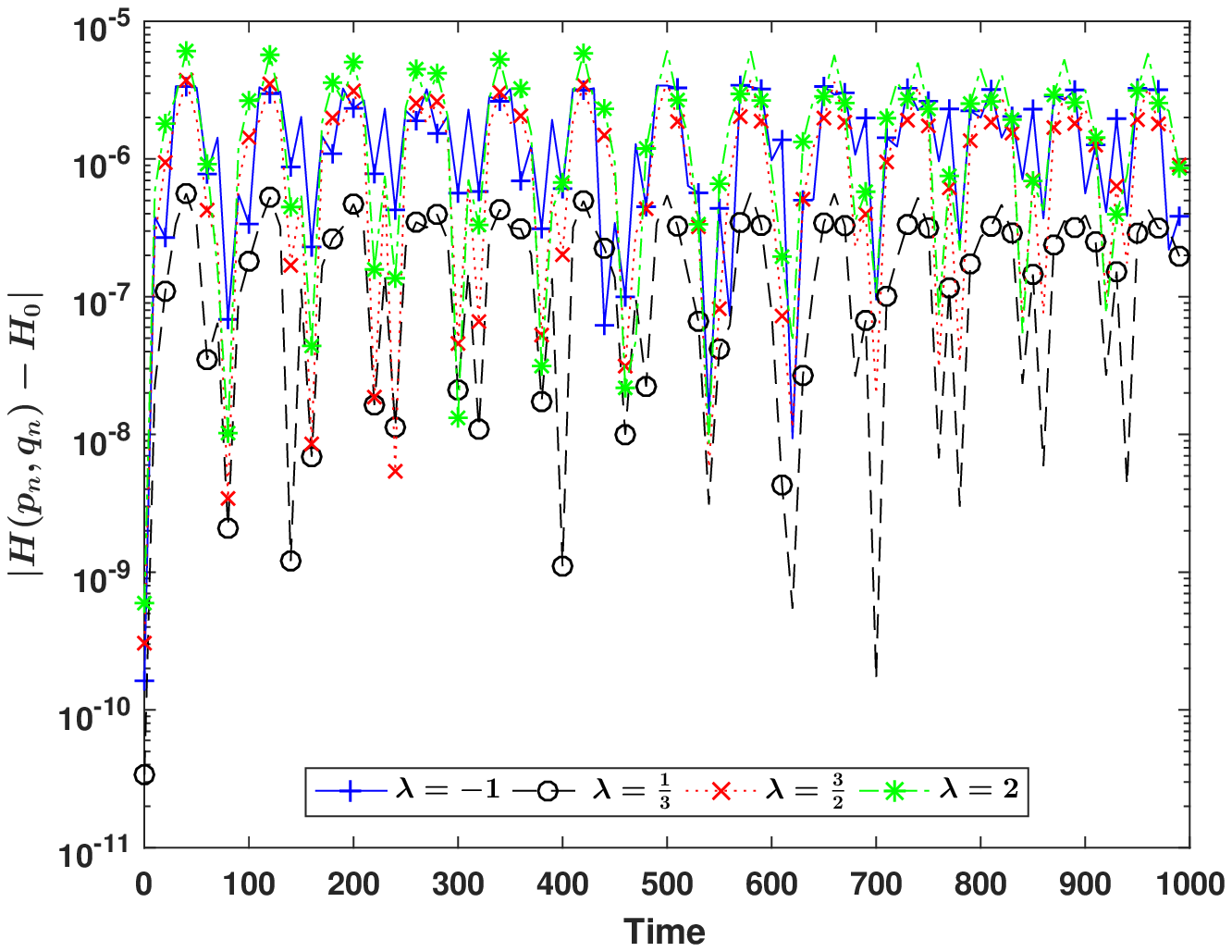}
\includegraphics[width=50mm]{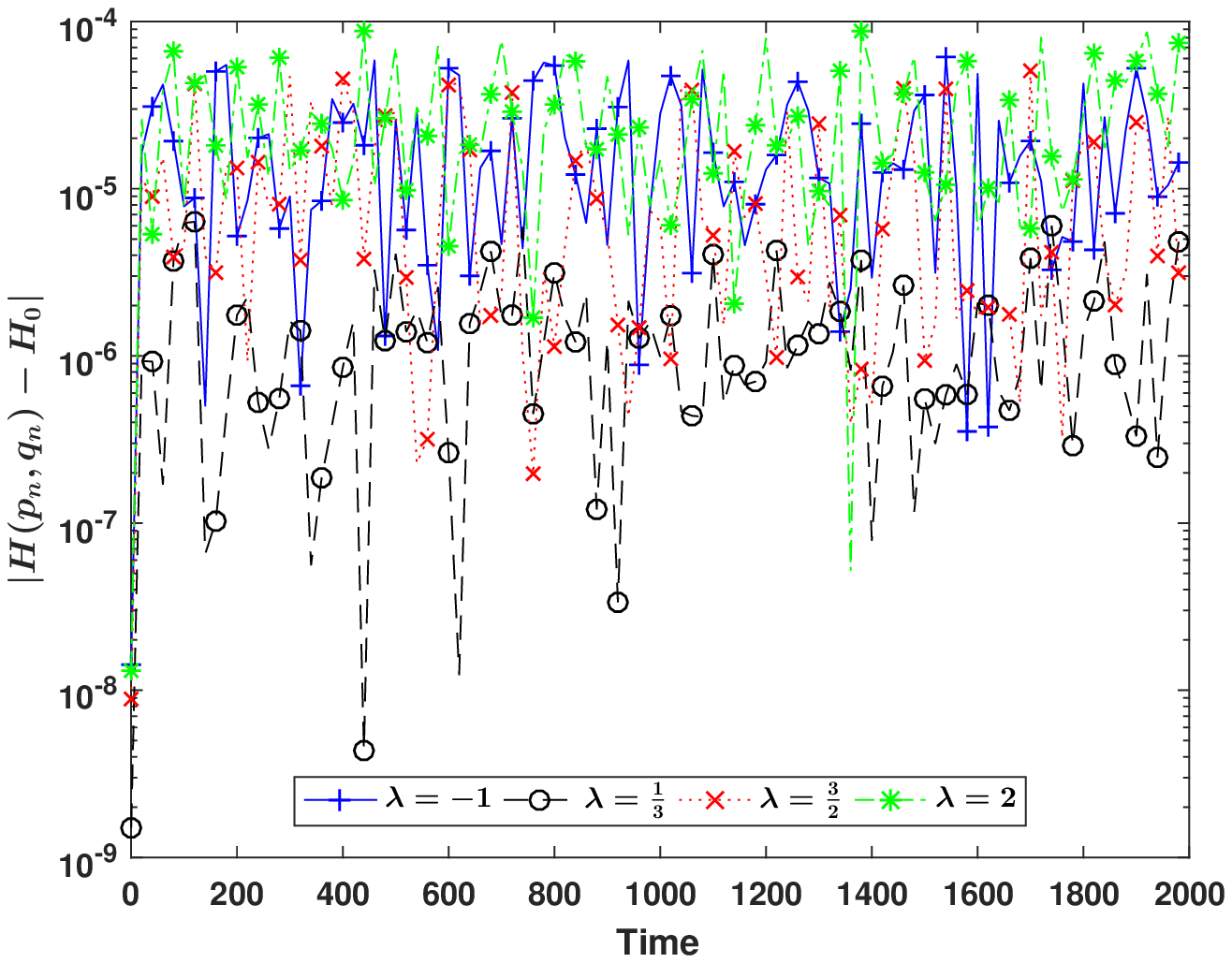}
\end{minipage}}
\caption{Energy conservation of Schemes \Rmnum{1}--\Rmnum{3} for the H\'{e}non-Heiles model with box orbits (first row) and chaotic orbits (second row).}\label{Fig-2}
\end{figure}

\begin{table}[H]
\tabcolsep=8.8pt
\small \renewcommand \arraystretch{1.25} \centering
\caption{Convergence tests of Schemes \Rmnum{1}--\Rmnum{3} for the H\'{e}non-Heiles model with chaotic orbits.}\label{Tab-1} 
\begin{tabularx}{\textwidth}{*{10}{l}} \toprule
\multirow{2}*{Scheme} & \multirow{2}*{$h$} & \multicolumn{2}{c}{$\lambda=-1$} & \multicolumn{2}{c}{$\lambda=\frac{1}{3}$} & \multicolumn{2}{c}{$\lambda=\frac{3}{2}$} & \multicolumn{2}{c}{$\lambda=2$}\\
\cmidrule(lr){3-4}\cmidrule(lr){5-6}\cmidrule(lr){7-8}\cmidrule(lr){9-10}
&& Error & Order & Error & Order & Error & Order & Error & Order\\ \midrule
\multirow{4}*{\centering \Rmnum{1}}
& 0.02/2 & 4.31E-03 & - & 4.38E-04 & - & 2.63E-03 & - & 3.85E-03 & -\\
& 0.02/4 & 2.09E-03 & 1.04 & 2.23E-04 & 0.98 & 1.34E-03 & 0.98 & 1.98E-03 & 0.96\\
& 0.02/8 & 1.03E-03 & 1.02 & 1.12E-04 & 0.99 & 6.73E-04 & 0.99 & 1.00E-03 & 0.98\\
& 0.02/16 & 5.13E-04 & 1.01 & 5.63E-05 & 0.99 & 3.38E-04 & 0.99 & 5.06E-04 & 0.99\\ \midrule
\multirow{4}*{\Rmnum{2}}
& 0.02/2 & 5.09E-04 & - & 2.35E-05 & - & 2.24E-04 & - & 4.77E-04 & -\\
& 0.02/4 & 1.25E-04 & 2.02 & 5.87E-06 & 2.00 & 5.67E-05 & 1.98 & 1.21E-04 & 1.98\\
& 0.02/8 & 3.11E-05 & 2.01 & 1.47E-06 & 2.00 & 1.42E-05 & 1.99 & 3.06E-05 & 1.99\\
& 0.02/16 & 7.74E-06 & 2.01 & 3.67E-07 & 2.00 & 3.57E-06 & 2.00 & 7.68E-06 & 1.99\\ \midrule
\multirow{4}*{\Rmnum{3}}
& 0.02/2 & 1.12E-04 & - & 3.54E-06 & - & 5.04E-05 & - & 1.18E-04 & -\\
& 0.02/4 & 2.79E-05 & 2.00 & 8.84E-07 & 2.00 & 1.26E-05 & 2.00 & 2.94E-05 & 2.00\\
& 0.02/8 & 6.97E-06 & 2.00 & 2.21E-07 & 2.00 & 3.15E-06 & 2.00 & 7.34E-06 & 2.00\\
& 0.02/16 & 1.74E-06 & 2.00 & 5.52E-08 & 2.00 & 7.88E-07 & 2.00 & 1.84E-06 & 2.00\\ \bottomrule
\end{tabularx}
\end{table}
\vspace{-2.5mm}
\begin{figure}[H]
\centering
\subfigure[AVF (first column)]{
\begin{minipage}[t]{0.32\textwidth}
\centering
\includegraphics[width=50mm]{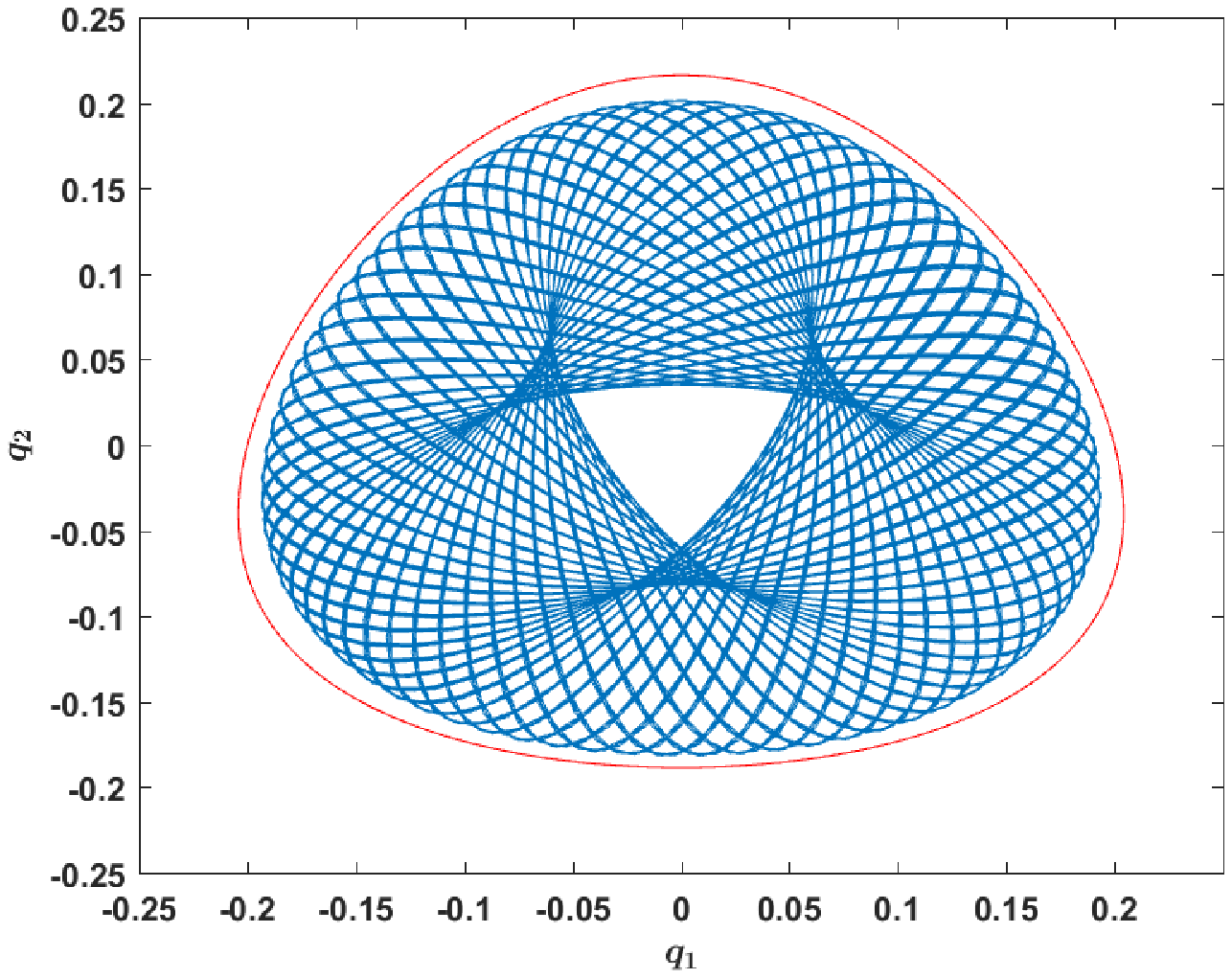}
\includegraphics[width=50mm]{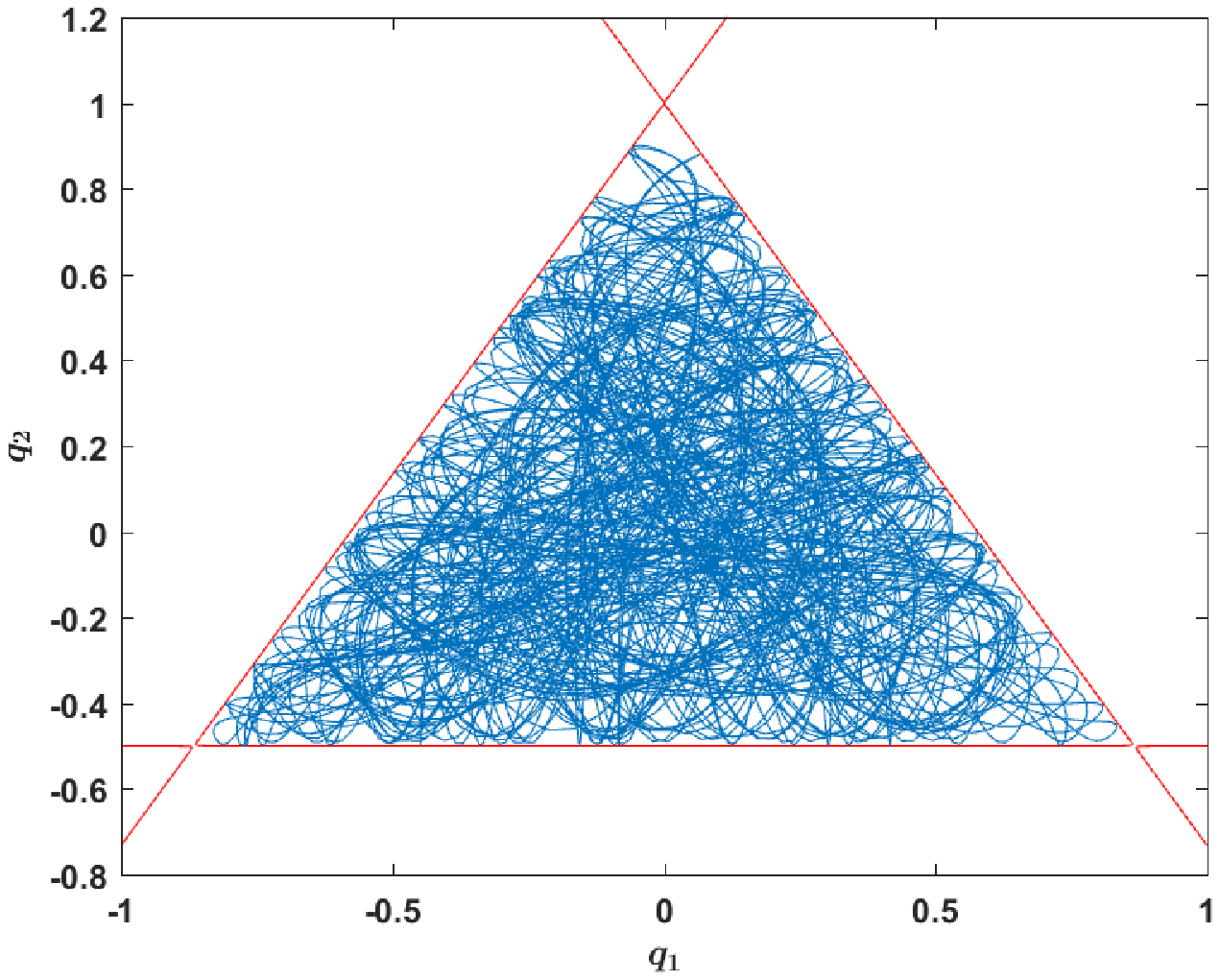}
\end{minipage}}\hspace{-6.5mm}
\subfigure[Scheme \Rmnum{4} (second column)]{
\begin{minipage}[t]{0.32\textwidth}
\centering
\includegraphics[width=50mm]{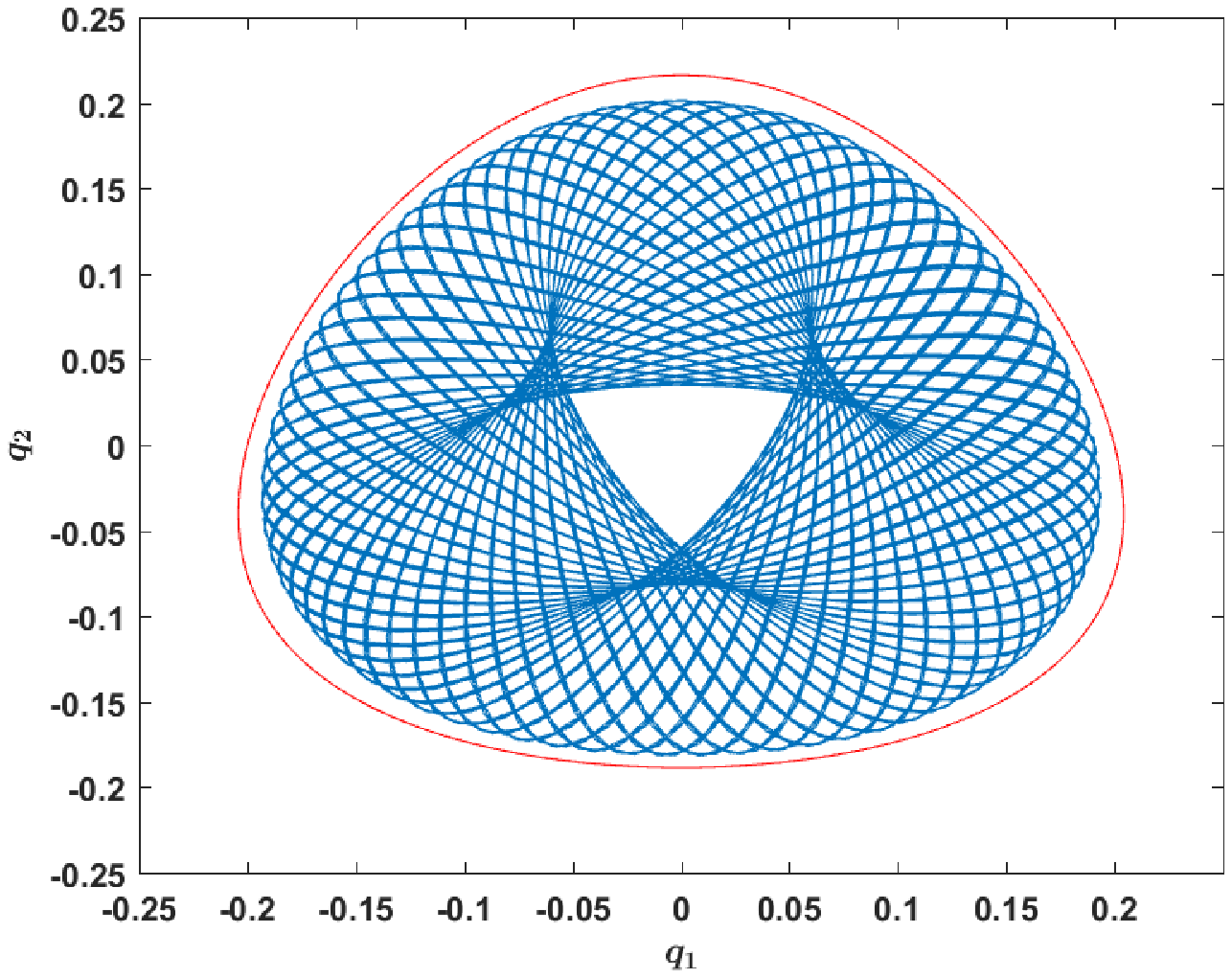}
\includegraphics[width=50mm]{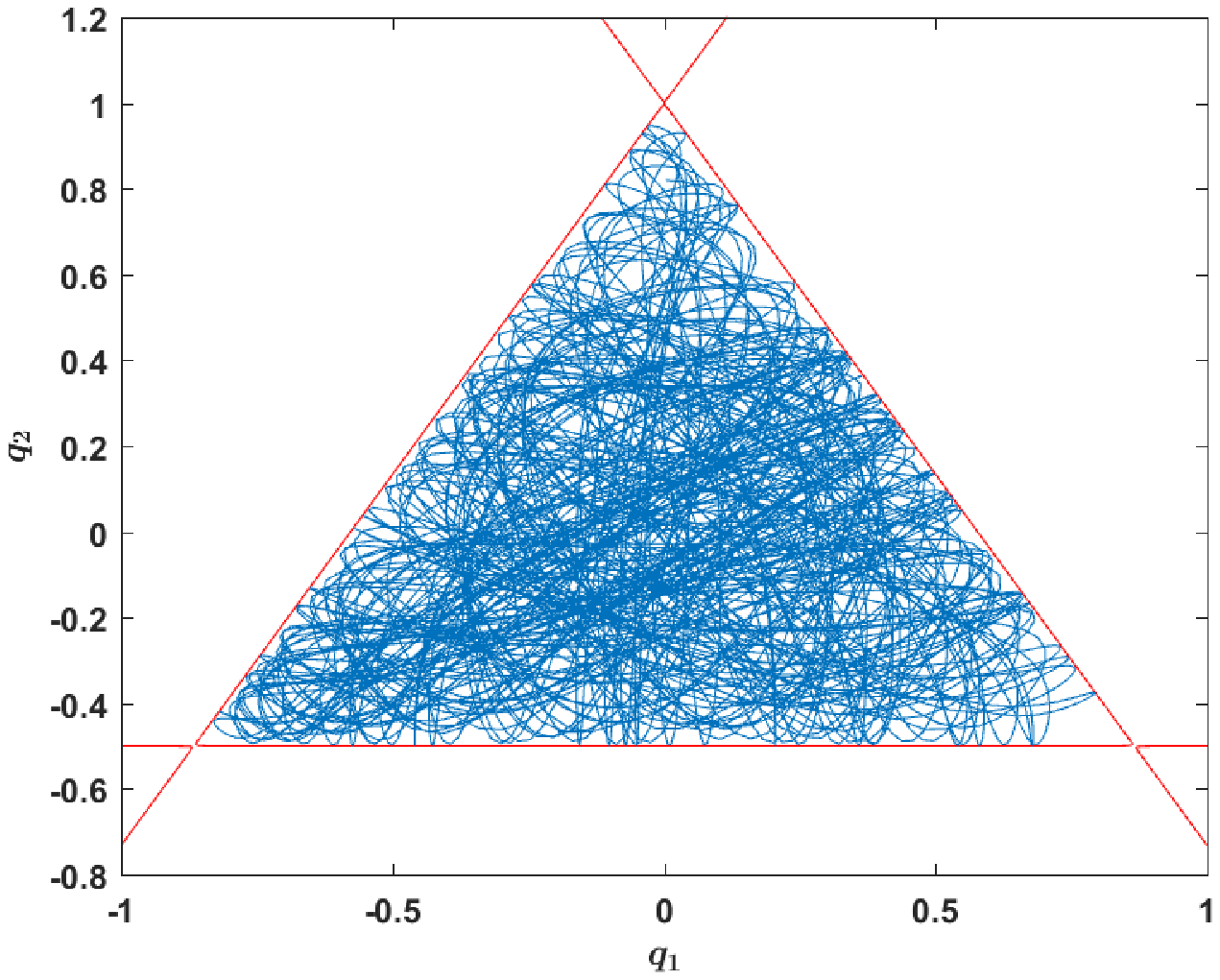}
\end{minipage}}\hspace{-6.5mm}
\subfigure[Scheme \Rmnum{5} (third column)]{
\begin{minipage}[t]{0.32\textwidth}
\centering
\includegraphics[width=50mm]{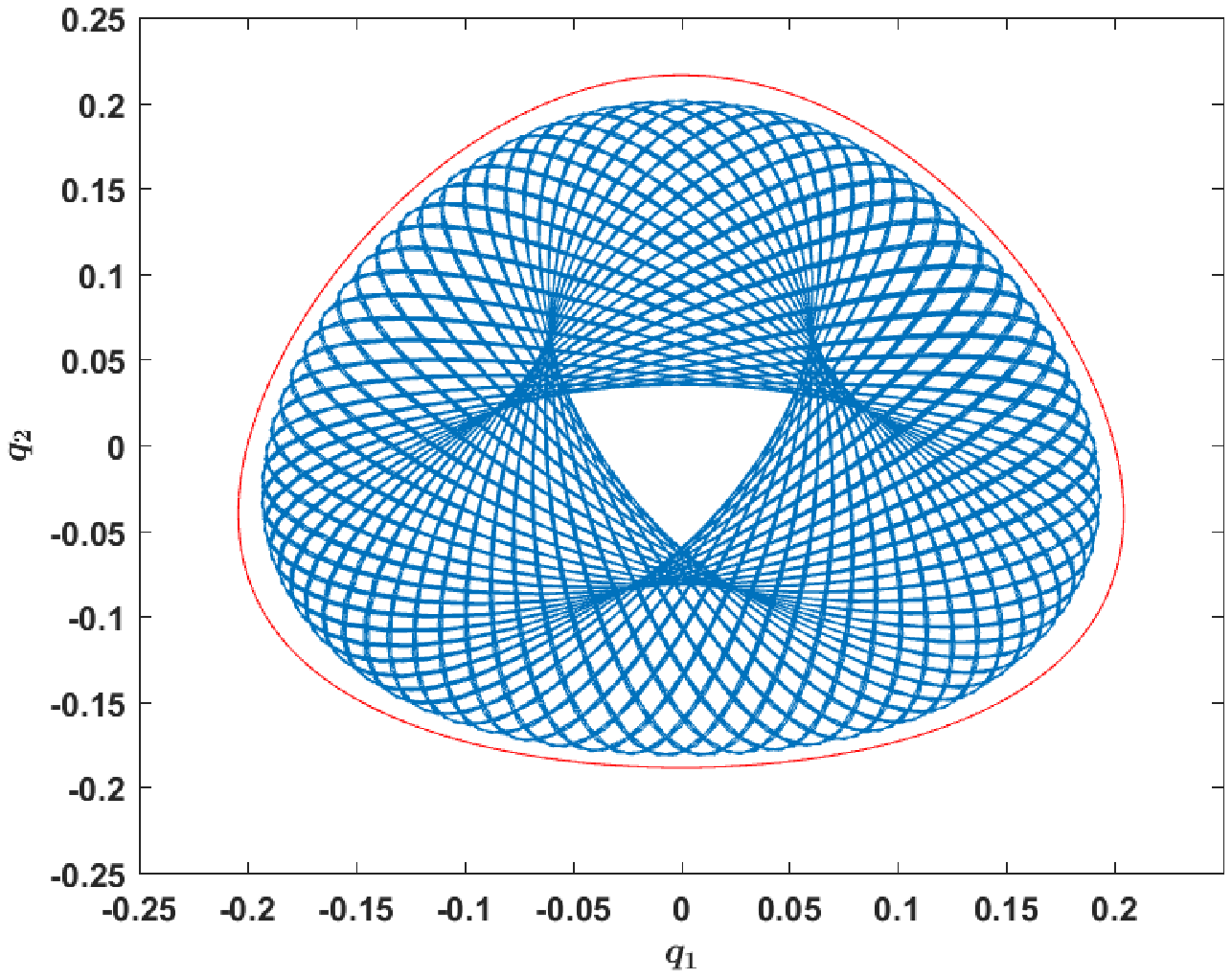}
\includegraphics[width=50mm]{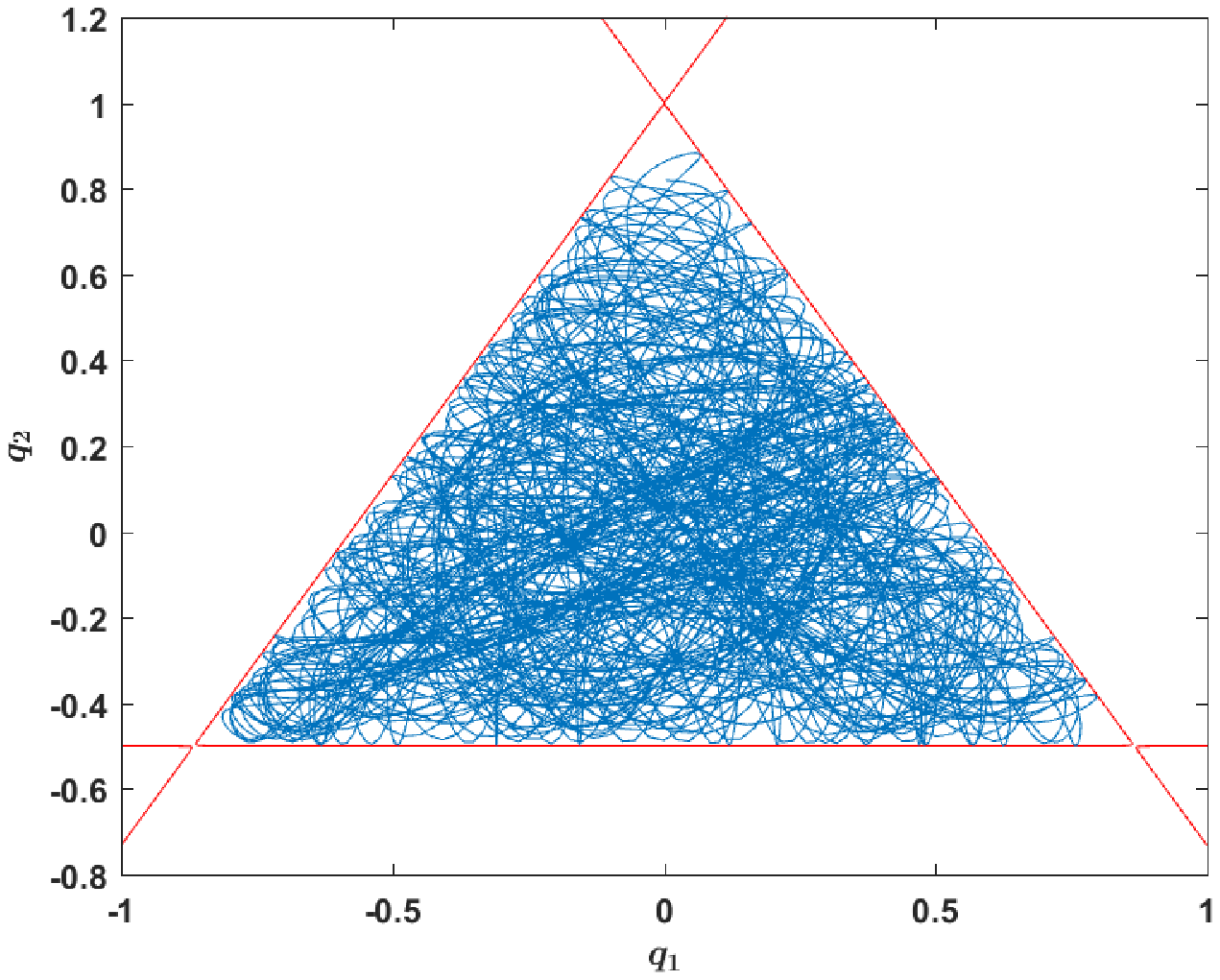}
\end{minipage}}\vspace{-2mm}
\caption{Box orbits (first row) and chaotic orbits (second row) of three energy-preserving schemes for the H\'{e}non-Heiles model till $t=1000$.}\label{Fig-3}
\end{figure}
\vspace{-2.5mm}
\begin{table}[H]
\tabcolsep=19pt
\small \renewcommand\arraystretch{1.25} \centering
\caption{Convergence tests of three energy-preserving schemes for the H\'{e}non-Heiles model with chaotic orbits.}\label{Tab-2}
\begin{tabularx}{\textwidth}{*{7}{l}} \toprule
\multirow{2}*{$h$} & \multicolumn{2}{c}{AVF} & \multicolumn{2}{c}{Scheme \Rmnum{4}} & \multicolumn{2}{c}{Scheme \Rmnum{5}}\\
\cmidrule(lr){2-3}\cmidrule(lr){4-5}\cmidrule(lr){6-7}
& Error & Order & Error & Order & Error & Order\\ \midrule
0.02/2 & 1.78E-05 & - & 2.50E-03 & - & 4.23E-06 & -\\
0.02/4 & 4.46E-06 & 2.00 & 1.24E-03 & 1.01 & 1.06E-06 & 2.00\\
0.02/8 & 1.11E-06 & 2.00 & 6.22E-04 & 0.99 & 2.64E-07 & 2.00\\
0.02/16 & 2.78E-07 & 2.00 & 3.10E-04 & 1.00 & 6.60E-08 & 2.00\\
\bottomrule
\end{tabularx}
\end{table}

\begin{figure}[H]
\centering
\includegraphics[width=60mm]{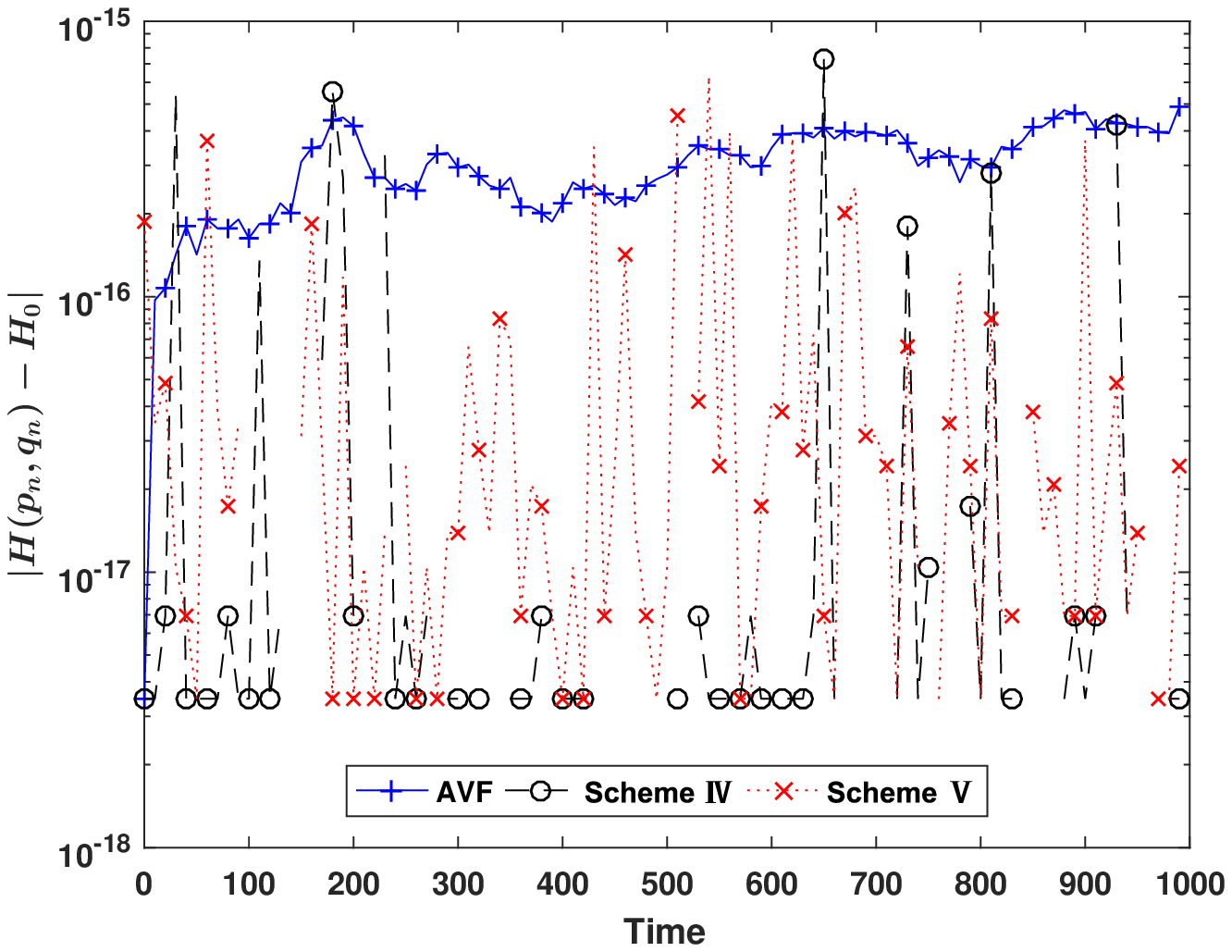}\hspace{-4.5mm}
\includegraphics[width=60mm]{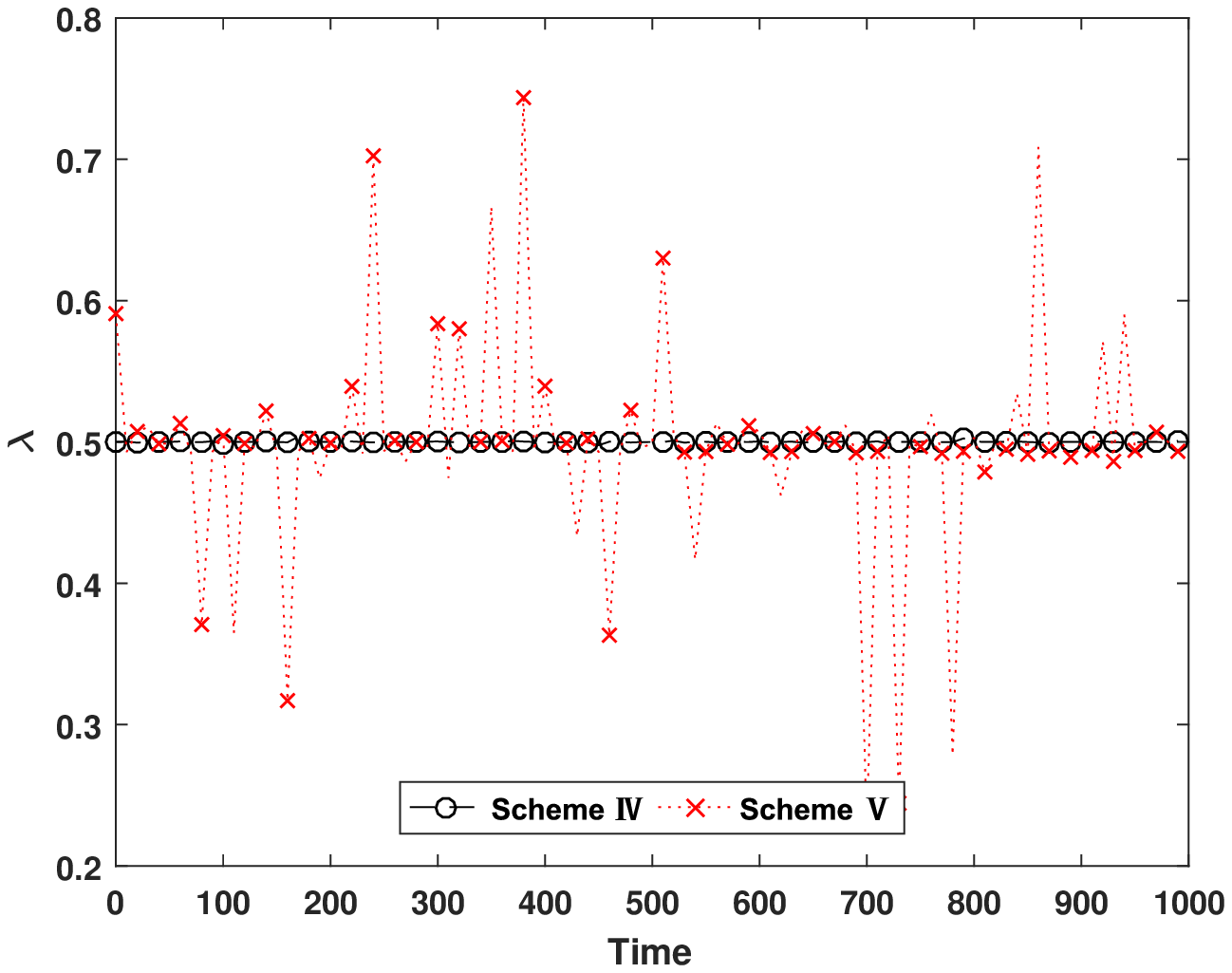}
\includegraphics[width=60mm]{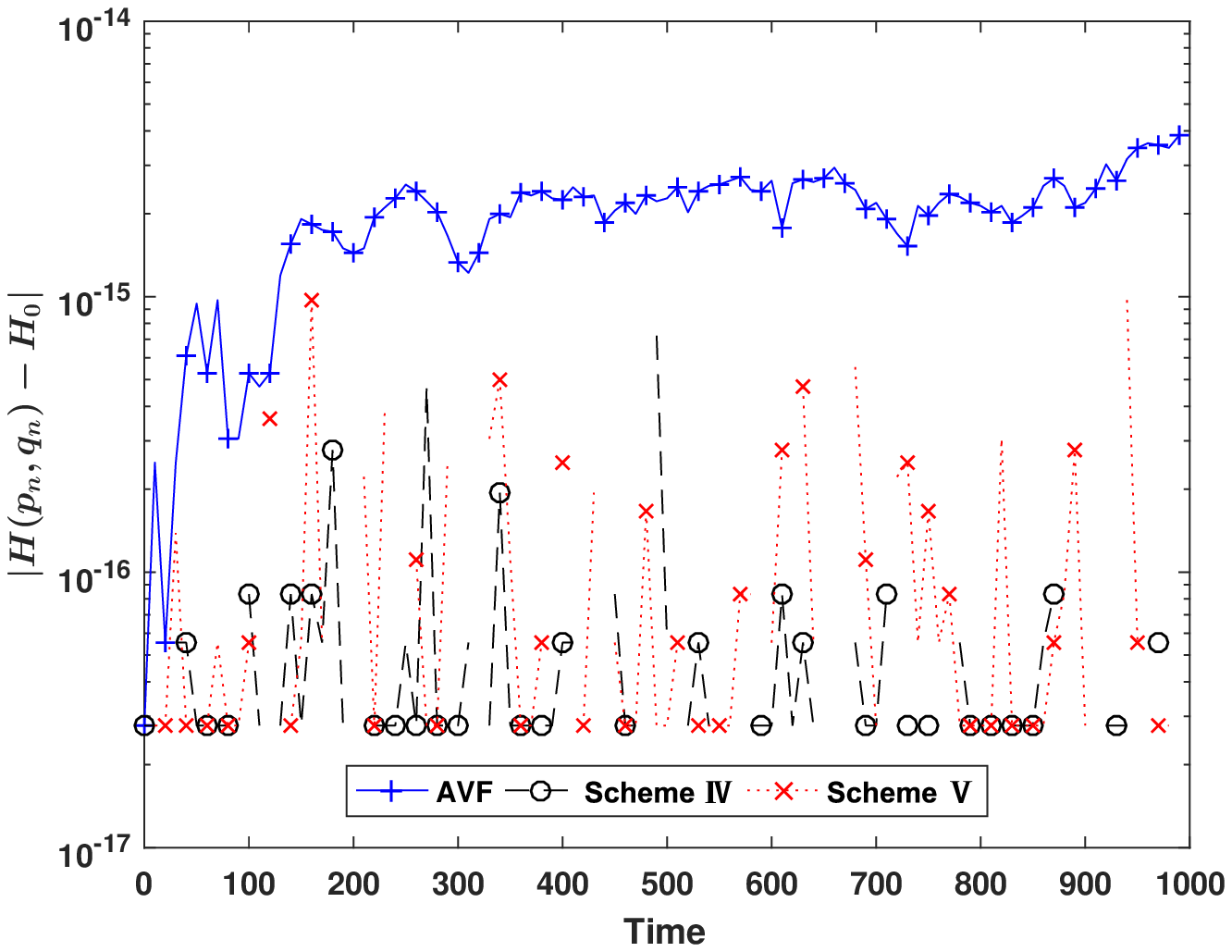}\hspace{-4.5mm}
\includegraphics[width=60mm]{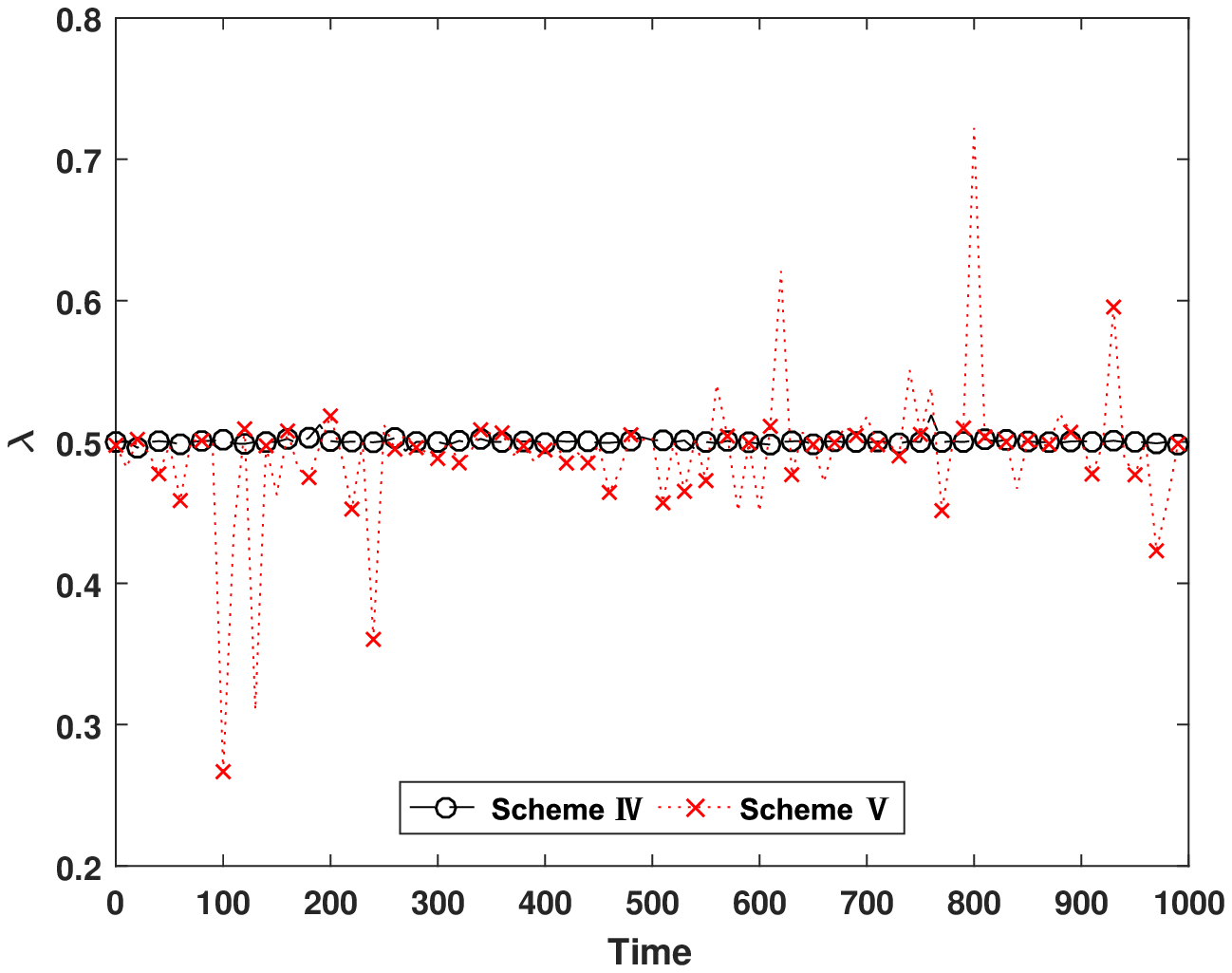}
\caption{Energy conservation (first column) of three energy-preserving schemes for the H\'{e}non-Heiles model; value distribution (second column) of the parameter in Schemes \Rmnum{4} and \Rmnum{5} with box orbits (first row) and chaotic orbits (second row).}\label{Fig-4}
\end{figure}

\subsection{The perturbed Kepler problem}\label{subS-5-2}
Our second experiment is the motion of a planet in the Schwarzschild potential for Einstein's general relativity theory. The Hamiltonian of the dynamics reads
\begin{equation}\label{eq-5-2}
H(p_1,p_2,q_1,q_2) = \frac{1}{2}(p_1^2+p_2^2)-\frac{1}{\sqrt{q_1^2+q_2^2}}-\frac{\mu}{3\sqrt{(q_1^2+q_2^2)^3}},
\end{equation}
where $\mu$ is a small number. In particular, when $\mu=0$, this system becomes the classical two-body problem. Moreover, the angular momentum $L(p_1,p_2,q_1,q_2) = q_1p_2-q_2p_1$ is a quadratic invariant. We choose the initial values \cite{hairer-06-geometric,luigi-12-EQUIP} as
\[p_1(0)=0,\ p_2(0)=\sqrt{\frac{1+e}{1-e}},\ q_1(0)=1-e,\ q_2(0)=0,\]
which confers an eccentricity $e$ on the orbit. Setting $e=0.6$ and $\mu=0.0075$, then $H(p,q)=H_0\approx-0.5391$, $L(p,q)=L_0=0.8$. Due to the complex form of \eqref{eq-5-2}, we do not consider Scheme \Rmnum{2} which contains the calculation of higher derivatives.

Schemes \Rmnum{1} and \Rmnum{3} plot satisfactory orbits similar to those of the original problem \cite{hairer-06-geometric}. Since the trajectories drawn by the two schemes are similar, we give the numerical solutions of Scheme \Rmnum{1} in Fig. \ref{Fig-5}. The main observation is that the energy error remains bounded and small in the long-time simulation, and the angular momentum is well preserved in Fig. \ref{Fig-6}. These benefit from the symplecticity of both schemes. 

The performance of three energy-preserving schemes is presented in Fig. \ref{Fig-7} and Fig. \ref{Fig-8}. All three schemes present the expected orbits. It is clear that the energy error grows linearly for the AVF method, and those of Schemes \Rmnum{4} and \Rmnum{5} are bounded over longer time intervals with smaller order of magnitude. The angular momentum is well conserved by Schemes \Rmnum{4} and \Rmnum{5} as explained in Remark \ref{Re-4-1}. These results show that the proposed EQUIP schemes are more advantageous than the AVF method, especially in the ability to preserve the Hamiltonian and the quadratic invariant. The values of $\lambda$ still fluctuate around $1/2$, which is consistent with the theoretical existence in Theorem \ref{Th-4-1}. As a part of this experiment, in Fig. \ref{Fig-9}, we demonstrate the behavior of three energy-preserving schemes for the two-body problem. The results once again confirm the theoretical analysis of this paper.

\begin{figure}[H]
\centering
\subfigure[$\lambda=-1$]{
\includegraphics[width=60mm]{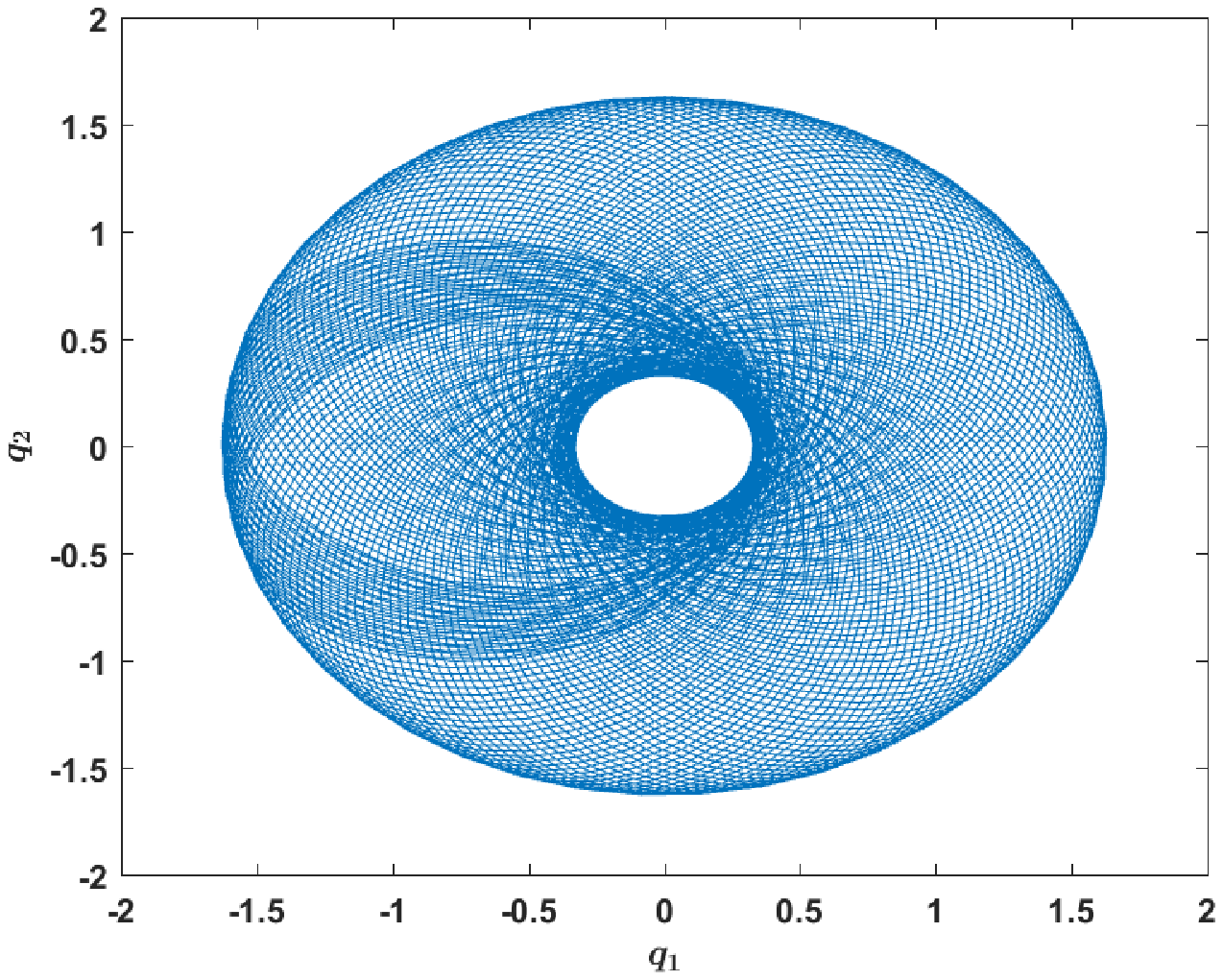}}\hspace{-4.5mm}
\subfigure[$\lambda=\frac{1}{3}$]{
\includegraphics[width=60mm]{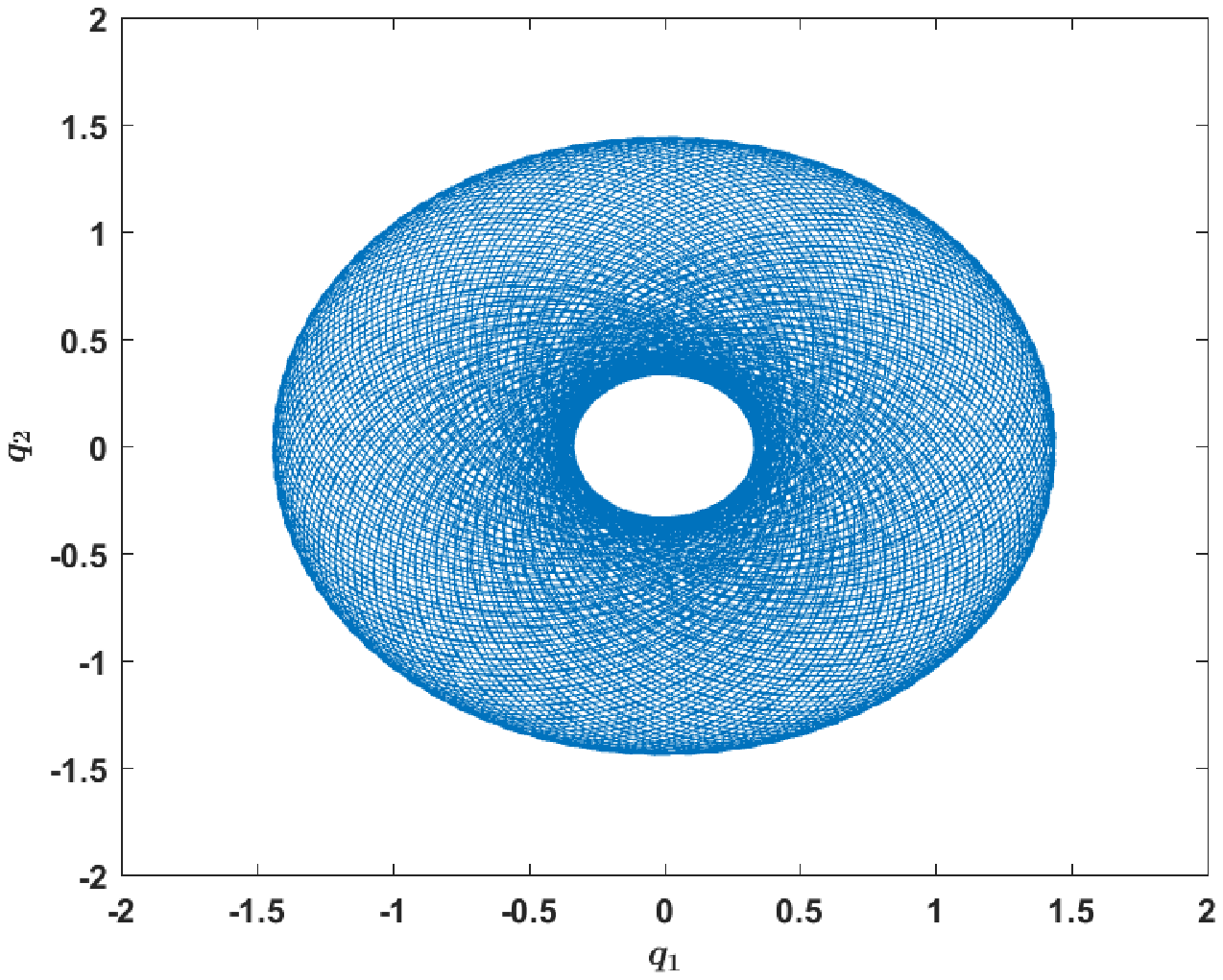}}\vspace{-2.5mm}
\subfigure[$\lambda=\frac{3}{2}$]{
\includegraphics[width=60mm]{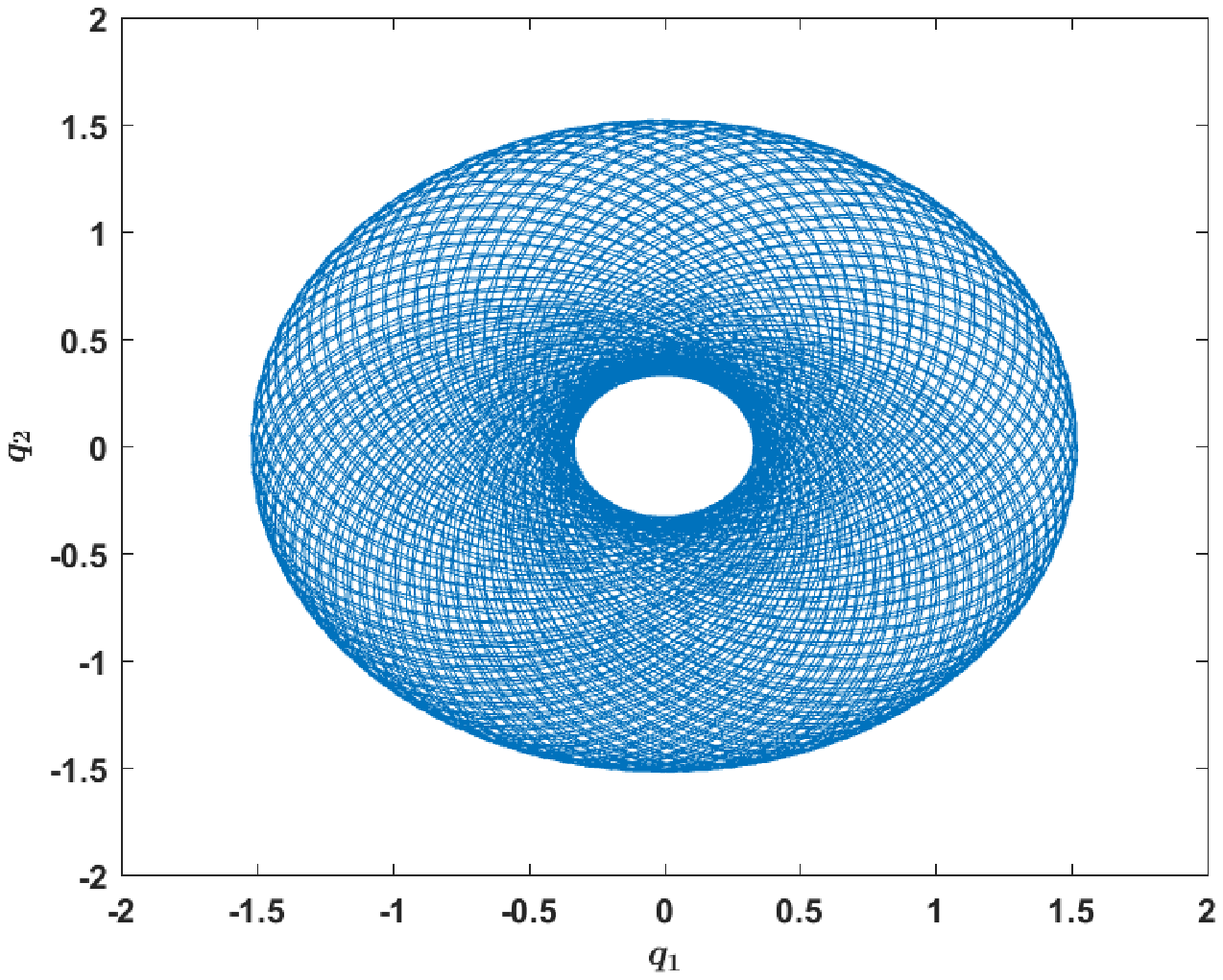}}\hspace{-4.5mm}
\subfigure[$\lambda=2$]{
\includegraphics[width=60mm]{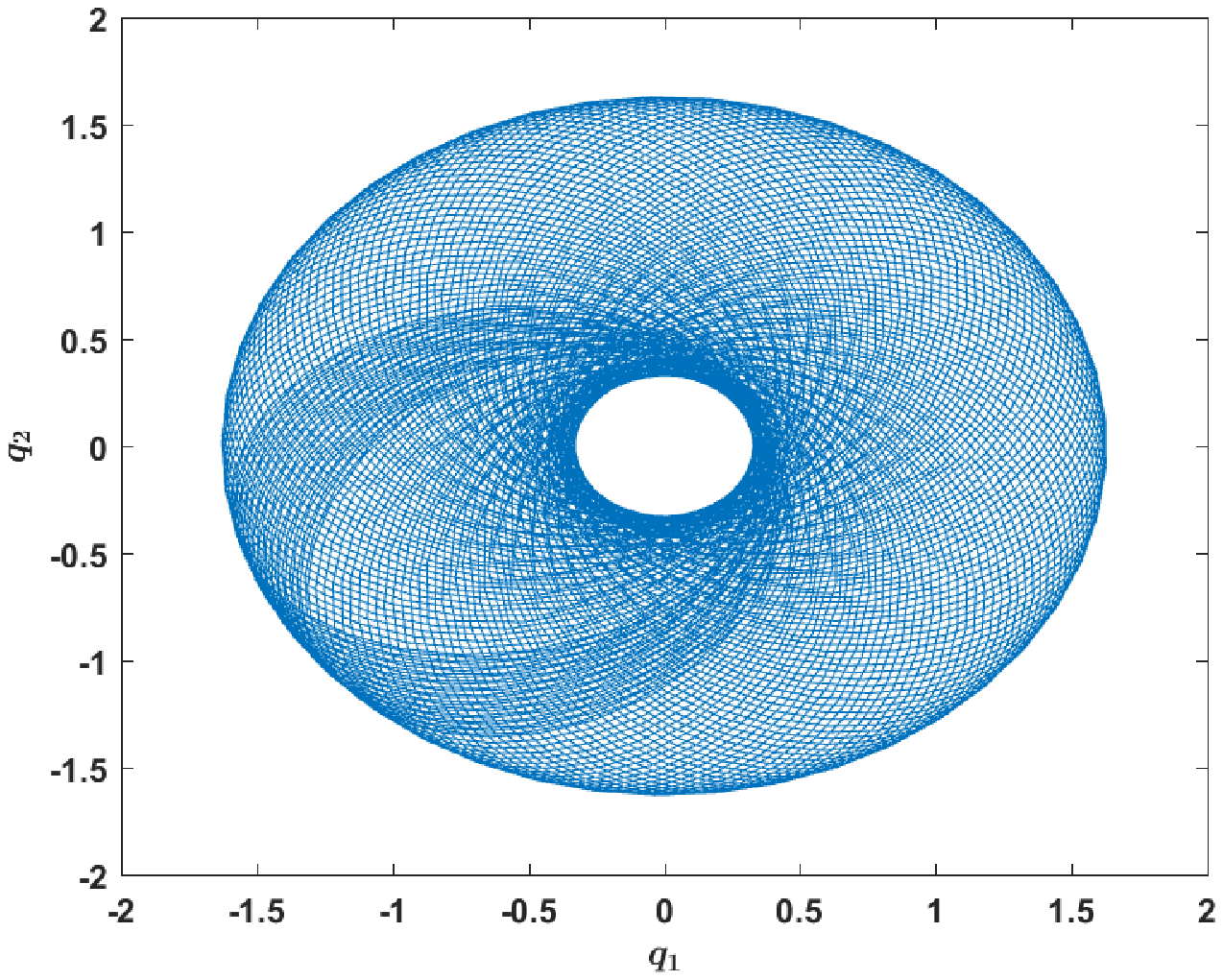}}
\caption{Numerical solutions of Scheme \Rmnum{1} for the perturbed Kepler problem till $t=1000$.}\label{Fig-5}
\end{figure}

\begin{figure}[H]
\centering
\includegraphics[width=60mm]{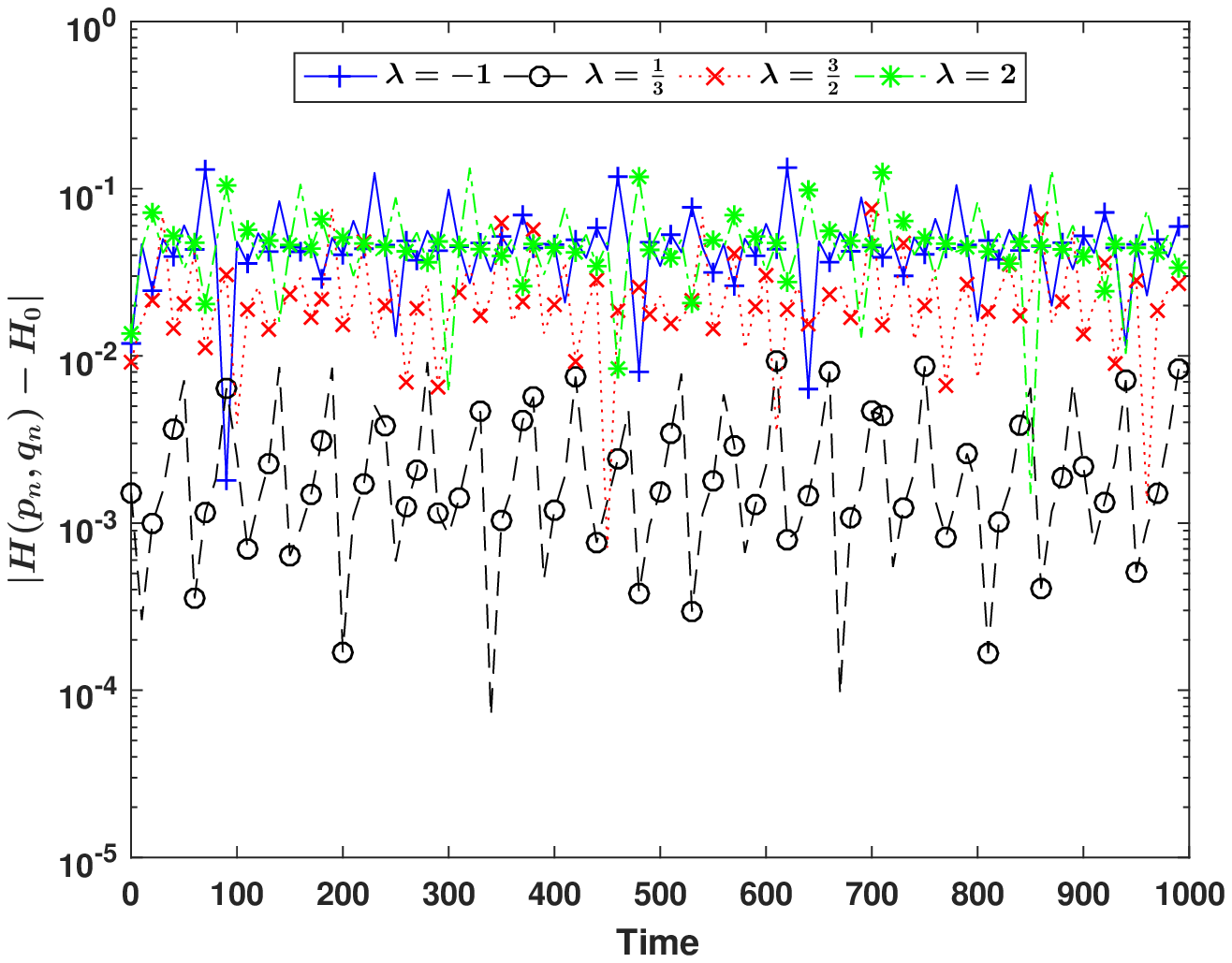}\hspace{-4.5mm}
\includegraphics[width=60mm]{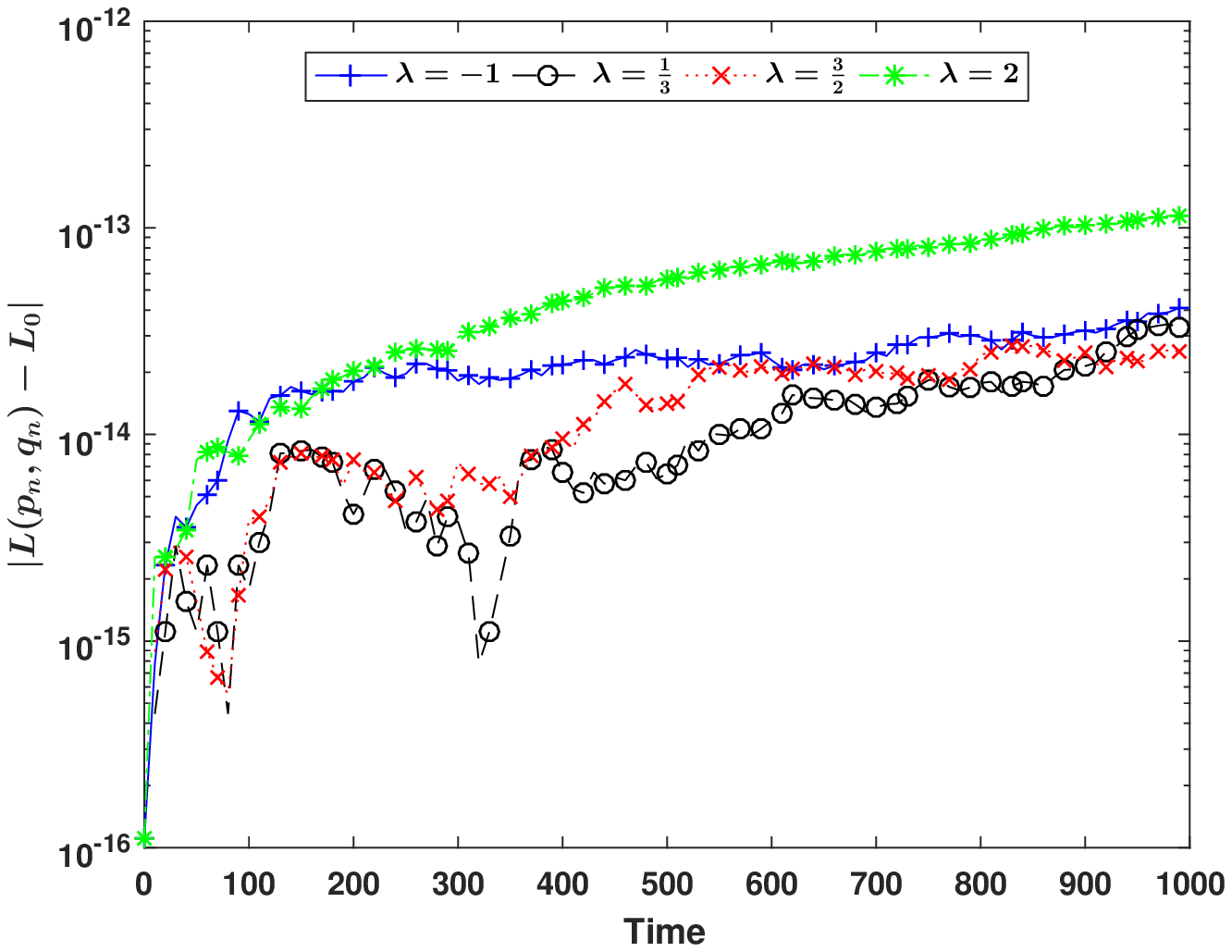}
\includegraphics[width=60mm]{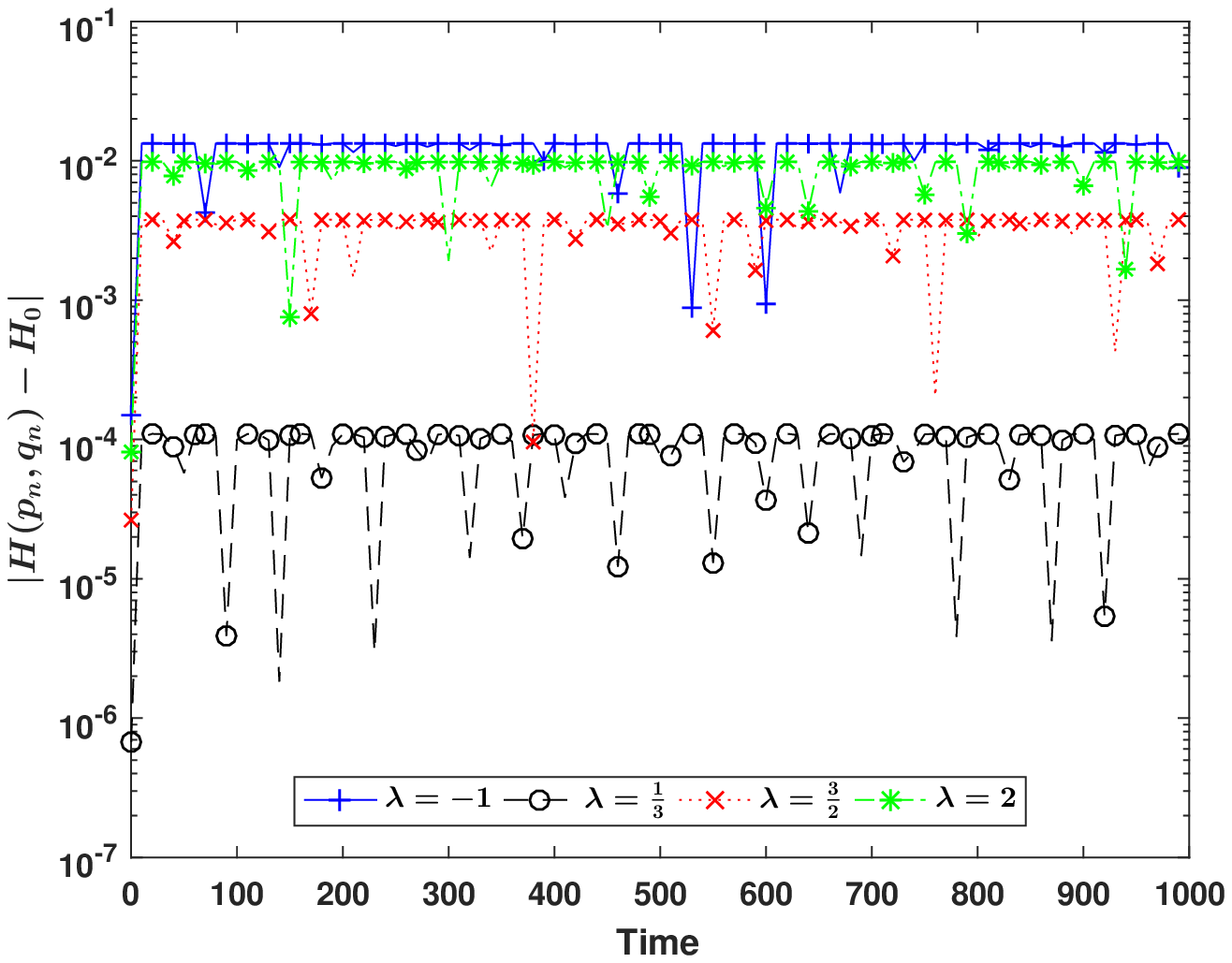}\hspace{-4.5mm}
\includegraphics[width=60mm]{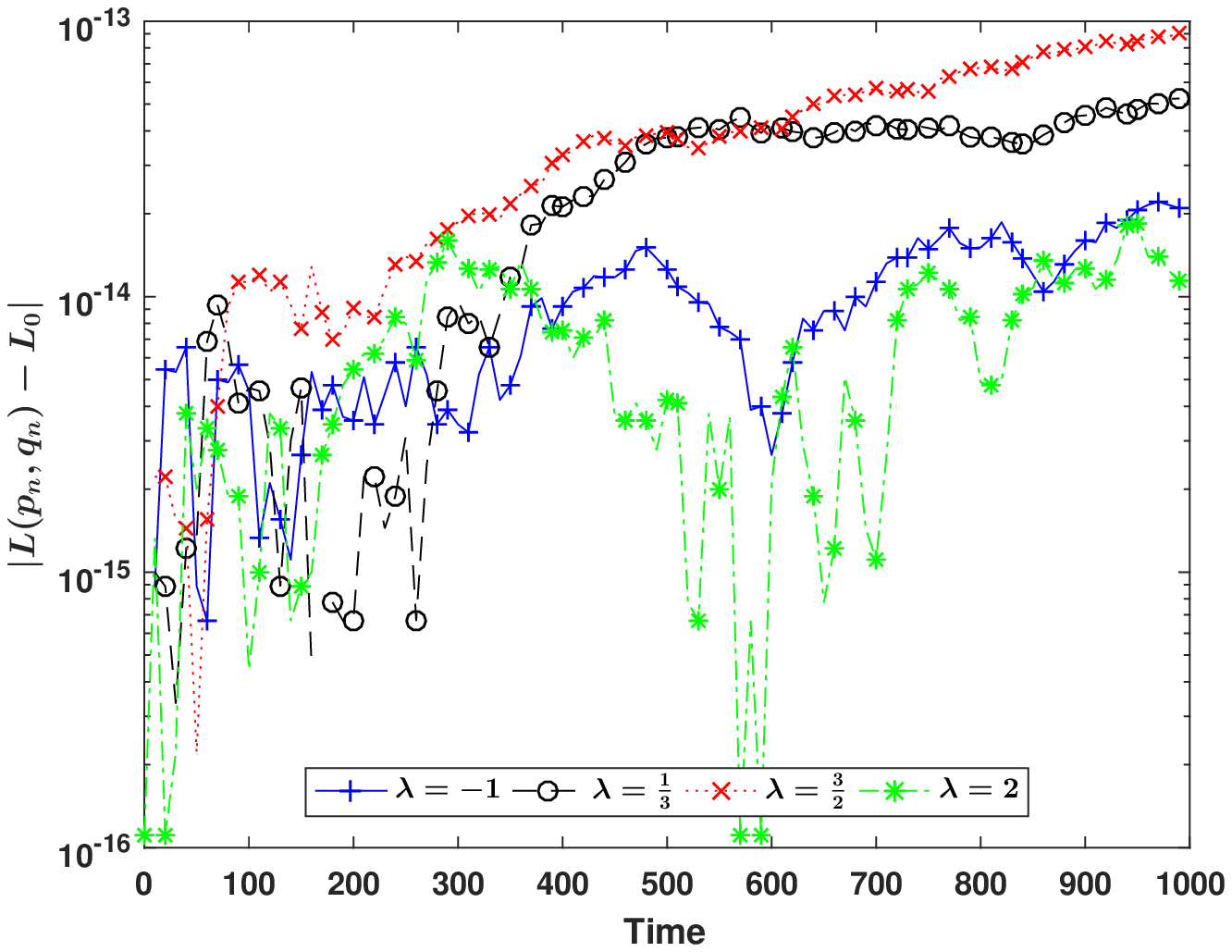}
\caption{Energy (left) and angular momentum (right) conservation of Schemes \Rmnum{1} (first row) and \Rmnum{3} (second row) for the perturbed Kepler problem.}\label{Fig-6}
\end{figure}

\begin{figure}[H]
\centering
\subfigure[AVF]{
\includegraphics[width=60mm]{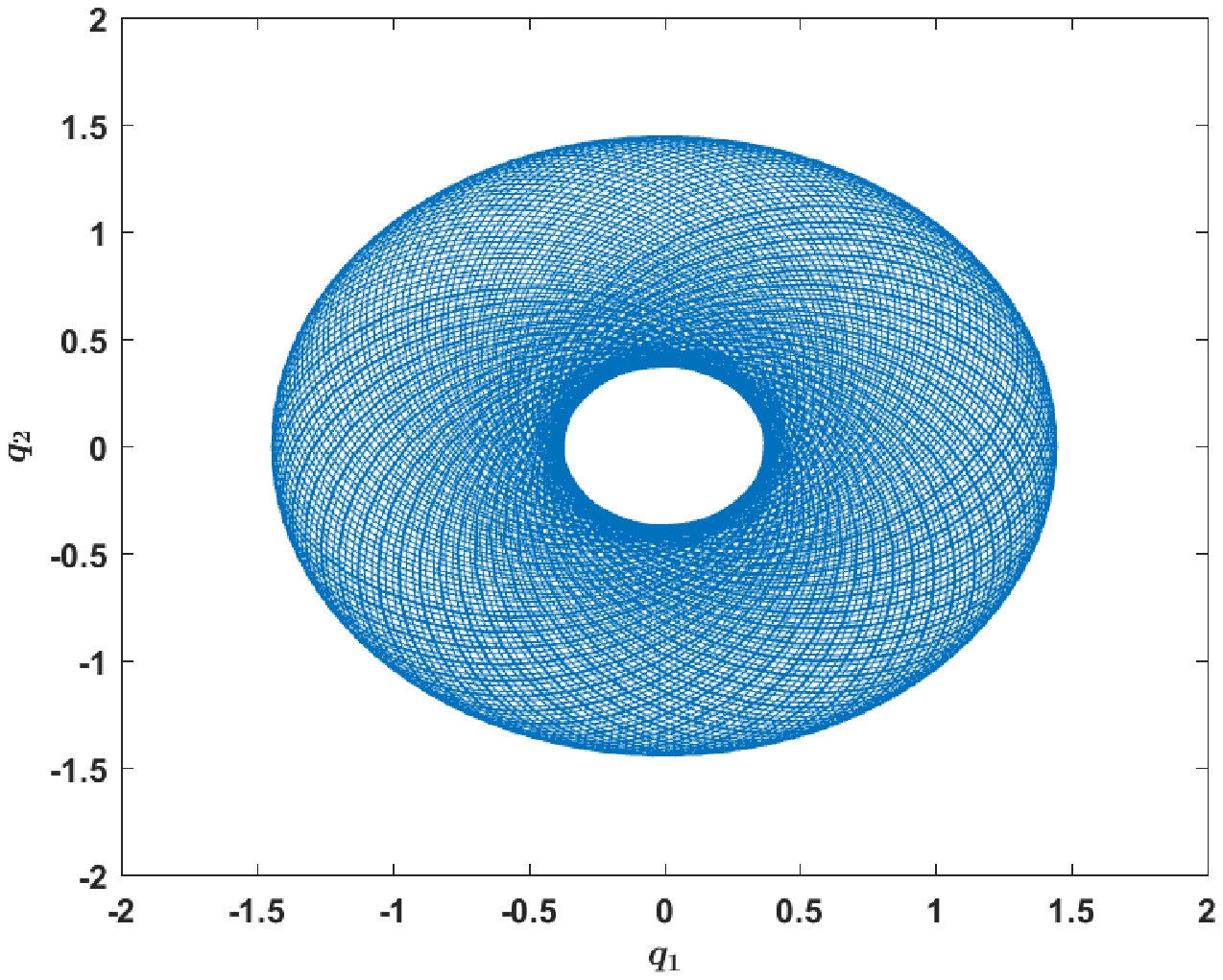}}\\
\subfigure[Scheme \Rmnum{4}]{
\includegraphics[width=60mm]{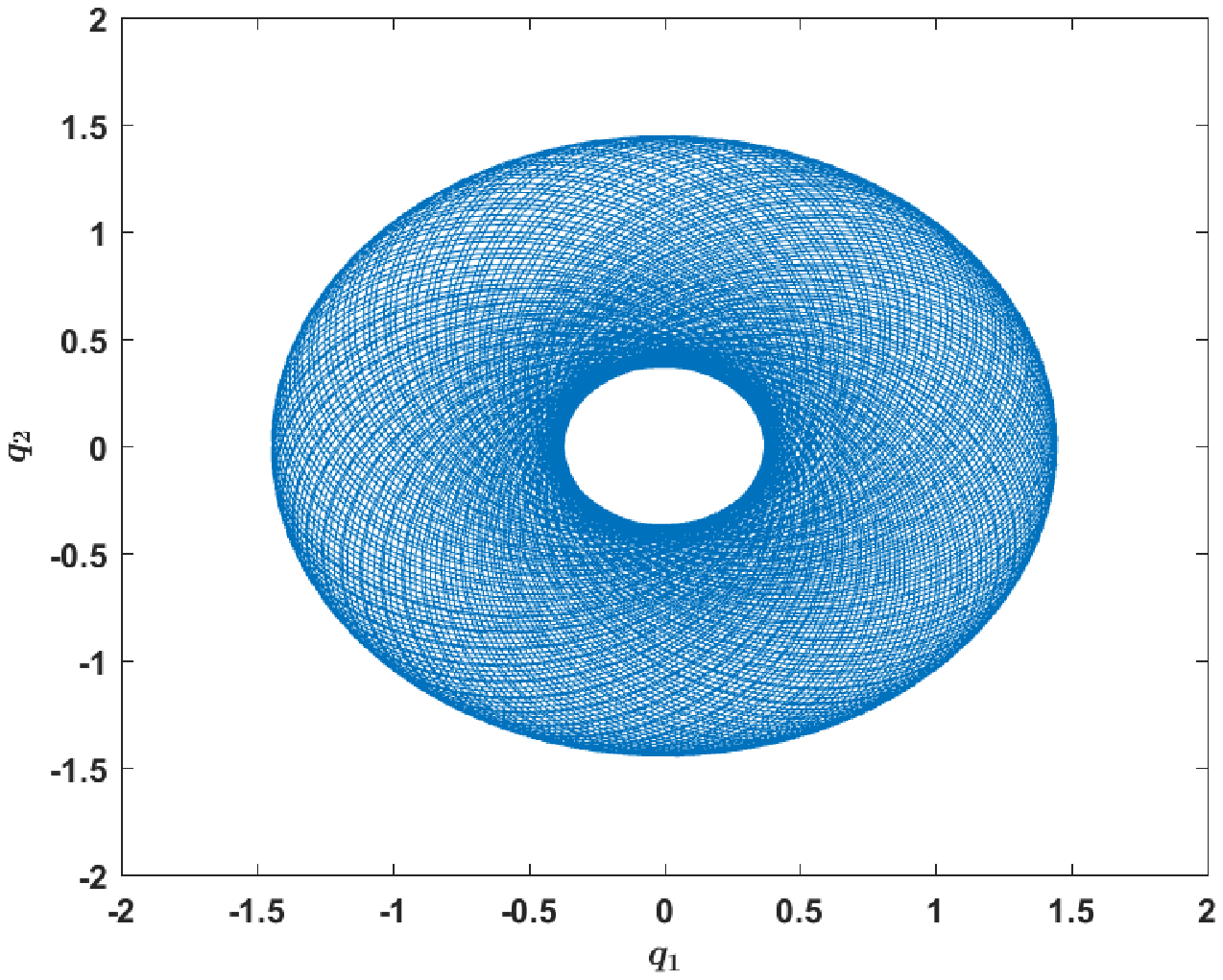}}\hspace{-4.5mm}
\subfigure[Scheme \Rmnum{5}]{
\includegraphics[width=60mm]{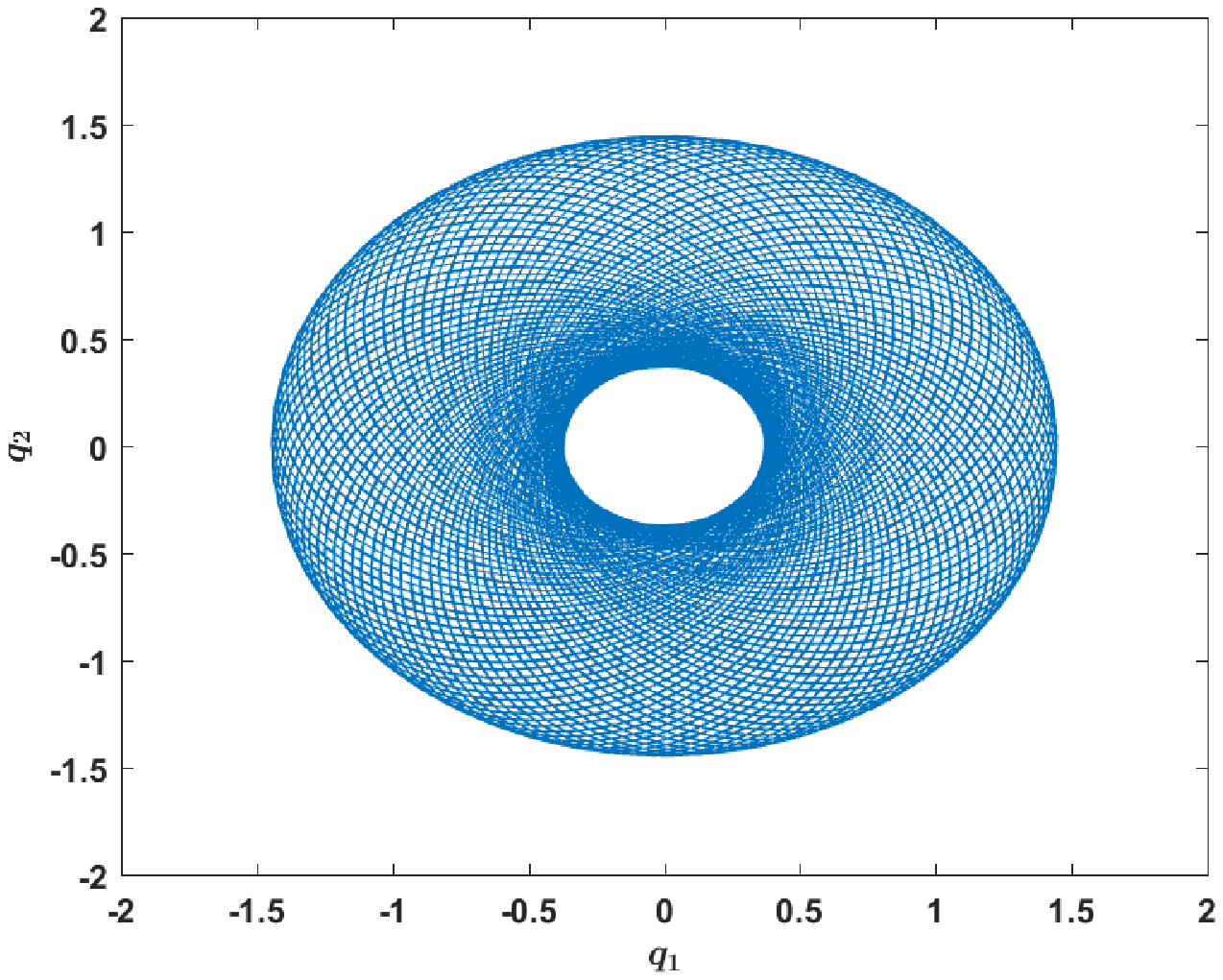}}
\caption{Numerical solutions of energy-preserving schemes for the perturbed Kepler problem till $t=1000$.}\label{Fig-7}
\end{figure}

\begin{figure}[H]
\centering
\includegraphics[width=60mm]{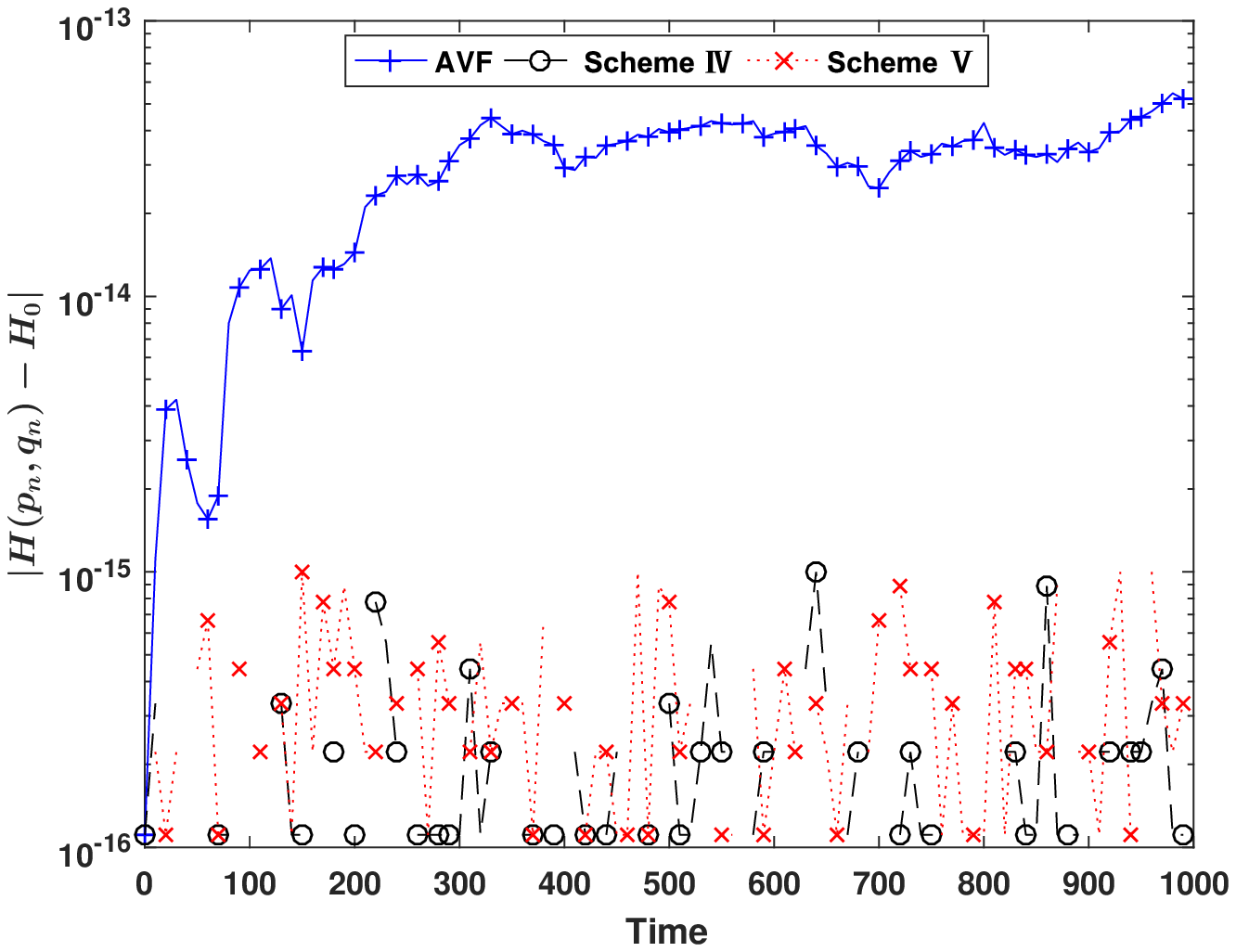}\hspace{-4.5mm}
\includegraphics[width=60mm]{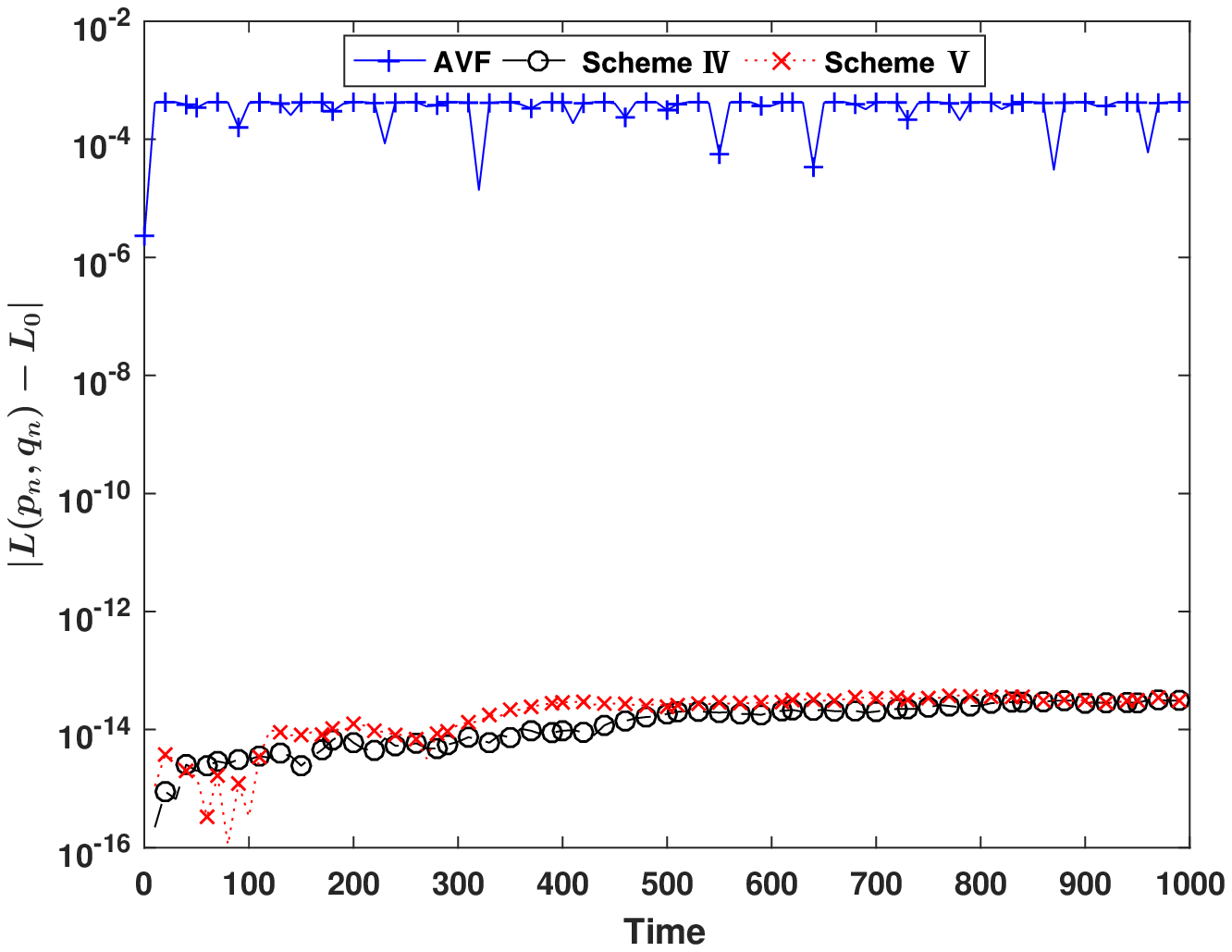}
\includegraphics[width=60mm]{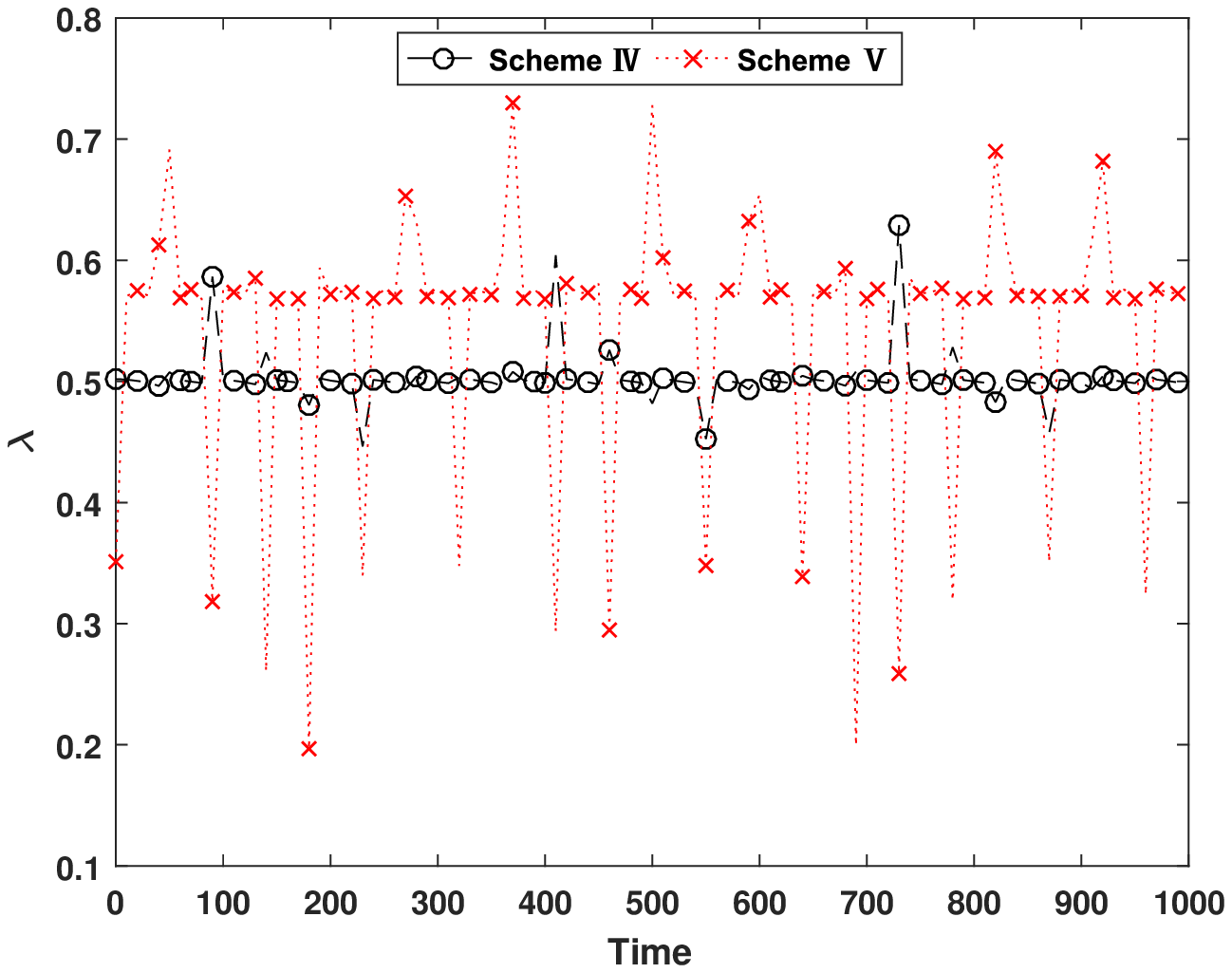}
\caption{Energy and angular momentum conservation (first row) of three energy-preserving schemes for the perturbed Kepler problem; value distribution (second row) of the parameter in Schemes \Rmnum{4} and \Rmnum{5}.}\label{Fig-8}
\end{figure}

\begin{figure}[H]
\centering
\includegraphics[width=60mm]{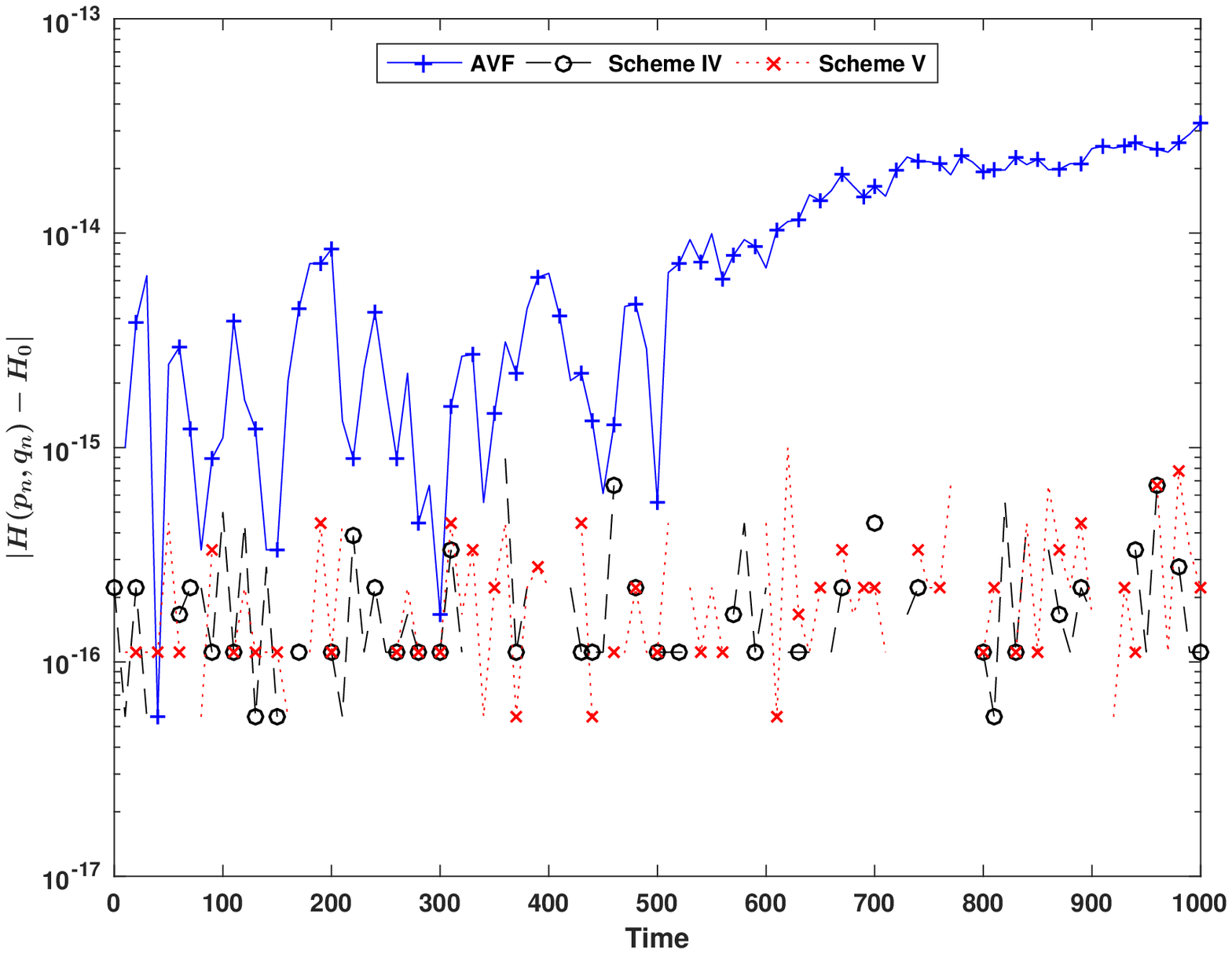}\hspace{-4.5mm}
\includegraphics[width=60mm]{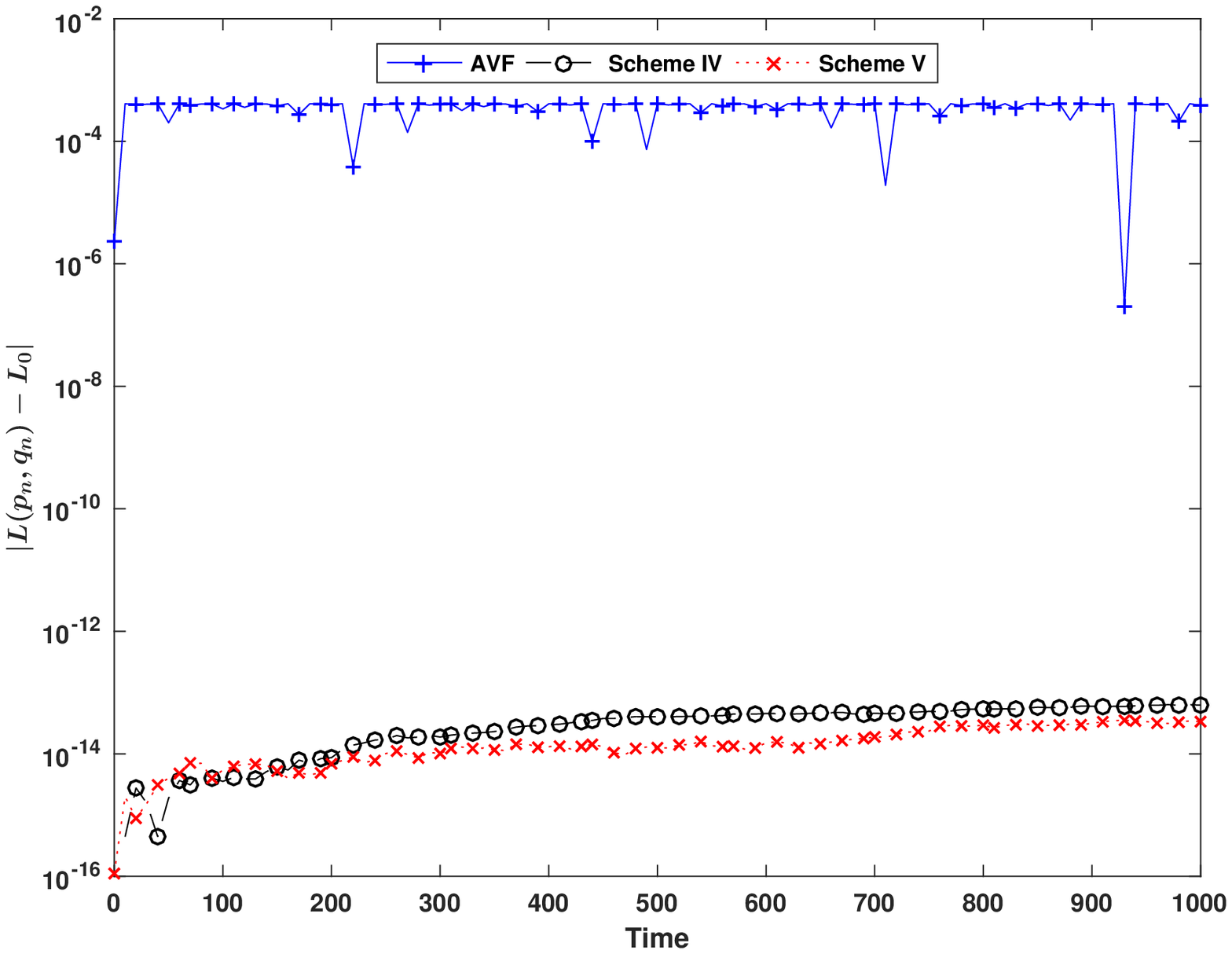}
\includegraphics[width=60mm]{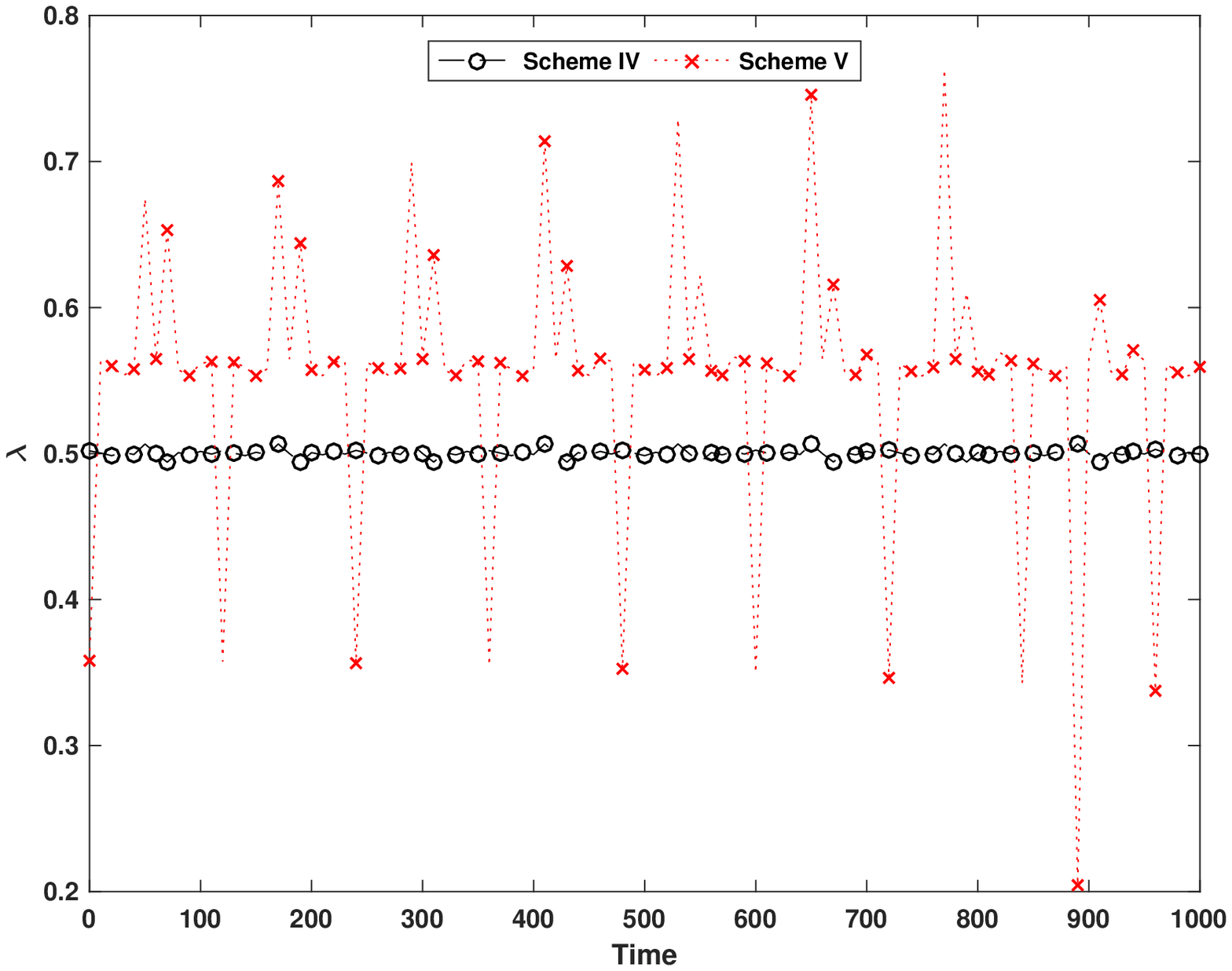}
\caption{Energy and angular momentum conservation (first row) of three energy-preserving schemes for the two-body problem; value distribution (second row) of the parameter in Schemes \Rmnum{4} and \Rmnum{5}.}\label{Fig-9}
\end{figure}

\section{Concluding remarks}\label{Sec-6}
We develop two new classes of arbitrary high-order structure-preserving methods for canonical Hamiltonian systems. The first class is two families of parameterized symplectic schemes. Moreover, this produces a simple symplectic scheme, namely Scheme \Rmnum{1}, which covers the symplectic Euler methods and the implicit midpoint rule. The second class is a novel family of arbitrary high-order energy and quadratic invariants preserving (EQUIP) schemes obtained by adjusting the free parameter of parameterized symplectic schemes. A more general form of generating functions is proposed, which generalizes the three classical generating functions that have been widely used. A rigorous proof shows that the parameter can be selected around $1/2$ to obtain the EQUIP methods. This also promotes an explanation why the implicit midpoint rule has excellent behavior in most situations. The long-time performances of all proposed schemes are demonstrated by numerical tests. 

The proposed schemes can be applied to the time discretization of Hamiltonian PDEs. The free parameter could be tuned to achieve the conservation of other non-quadratic invariants. The classes of symplectic integrators and variational integrators are identical. It is meaningful to discuss the parameterized symplectic schemes and their properties based on the discrete Hamilton's principle. We will further generalize these results to non-canonical Hamiltonian systems. 

\section*{Acknowledgements}
The authors wish to thank Professor Christian Lubich for his valuable suggestions by pointing out that the solved parameter in the EQUIP schemes depends on the initial points so that the obtained EQUIP schemes cannot completely inherit  the symplecticity. 

This work is supported by the National Key Research and Development Project of China (Grant No. 2018YFC1504205), the National Natural Science Foundation of China (Grant No. 11771213, 11971242), the Major Projects of Natural Sciences of University in Jiangsu Province of China (Grant No. 18KJA110003) and the Priority Academic Program Development of Jiangsu Higher Education Institutions.

%\section*{References}

\end{document}